\documentclass[12pt]{amsart}

\usepackage{
   a4wide,
  amsmath,
  amsfonts,
  amssymb,
  amsthm,
  graphicx, 
  times,
  color,
  hyphenat,
  stmaryrd, datetime, bbm, tikz-cd, thmtools, verbatim
}
\usepackage{ytableau}
\usepackage[shortlabels]{enumitem}

\usepackage{dynkin-diagrams}
\usepackage{mathdots}
\usepackage[all]{xy}
\usepackage{euscript,mathrsfs}
\usepackage{bbm,xspace}
\usepackage[final]{hyperref, bookmark}
\usepackage{cleveref}
\hypersetup{colorlinks=true, pdfstartview=FitV, linkcolor=blue, citecolor=blue, urlcolor=blue}

\usepackage[modulo]{lineno}
\usepackage{dirtytalk}

\usepackage{tikz}
\usetikzlibrary{cd}
\usetikzlibrary{decorations}
\usetikzlibrary{decorations.markings}
\usetikzlibrary{decorations.pathreplacing}
\usetikzlibrary{decorations.pathmorphing}
\usetikzlibrary{arrows.meta,shapes,positioning,matrix,calc}
\usetikzlibrary{shapes.callouts}
\tikzstyle{densely dotted}=[dash pattern=on \pgflinewidth off 0.5pt]
\tikzset{anchorbase/.style={baseline={([yshift=-0.5ex]current bounding box.center)}},
tinynodes/.style={font=\tiny,text height=0.25ex,text depth=0.05ex},
smallnodes/.style={font=\scriptsize,text height=0.75ex,text depth=0.15ex},
usual/.style={line width=0.9,color=black},
dusual/.style={line width=0.9,color=spinach,densely dashed},
pole/.style={line width=3.0,color=specialgray},
crossline/.style={preaction={draw=white,line width=5.0pt,-},preaction={draw=black,line width=0.9pt,-}},
crosspole/.style={preaction={draw=white,line width=6.0pt,-},preaction={draw=specialgray,line width=3.0pt,-}},
mor/.style={line width=0.75,color=black,fill=cream},
blob/.style={circle,fill,minimum size=5.0pt,inner sep=0pt,outer sep=0pt},
blobed/.style n args={3}{decoration={markings,post length=0.5mm,pre length=0.5mm,
mark=at position #1 with {\node[blob,#3,label=left:$#2\!$]at (0,0){};}
},postaction={decorate}},
rblobed/.style n args={3}{decoration={markings,post length=0.5mm,pre length=0.5mm,
mark=at position #1 with {\node[blob,#3,label=right:$\!#2$]at (0,0){};}
},postaction={decorate}},
cutline/.style={line width=2.25,color=purple,densely dotted},
}
\tikzstyle directed=[postaction={decorate,decoration={markings,mark=at position #1 with {\arrow[line width=0.25mm, black]{>}}}}]
\tikzstyle rdirected=[postaction={decorate,decoration={markings,mark=at position #1 with {\arrow[line width=0.25mm, black]{<}}}}]

\newcommand{\tikzdiagh}[2][]{\tikz[#1,thick,baseline={([yshift=1ex+#2]current bounding box.center)}]}
\newcommand{\tikzdiagc}[1][]{\tikzdiagh[#1]{-1ex}}

\tikzstyle{tikzdot}=[fill, circle, inner sep=2pt]
\tikzstyle{smoltikzdot}=[fill=white, draw=black, circle, inner sep=1pt]

\tikzset{
    partial ellipse/.style args={#1:#2:#3}{
        insert path={+ (#1:#3) arc (#1:#2:#3)}
    }
}

\input xy
\xyoption{all}
\definecolor{myblue}{rgb}{0,.5,1}
\definecolor{myred}{rgb}{0.9,0,0}
\definecolor{mygreen}{rgb}{0,0.7,0}

\def\down{\vee}
\def\up{\wedge}

\newcommand{\cP}{\mathcal{P}}

\DeclareMathOperator{\End}{End}

\DeclareMathOperator{\im}{im}

\newtheorem{mainthm}{Main Theorem}

\newtheorem{thm}{Theorem}[section]
\newtheorem{lem}[thm]{Lemma}
\newtheorem{cor}[thm]{Corollary}
\newtheorem{prop}[thm]{Proposition}
\newtheorem{exe}[thm]{Example}

\theoremstyle{definition}
\newtheorem{defn}[thm]{Definition}

\newtheorem{rem}[thm]{Remark}

\newtheorem*{rem*}{Remark}

\newcommand{\cO}{\mathcal{O}}

\DeclareMathOperator{\id}{Id}

\newcommand{\nnfootnote}[1]{%
\begin{NoHyper}
\renewcommand\thefootnote{}\footnote{#1}%
\addtocounter{footnote}{-1}%
\end{NoHyper}
}

\catcode`\@=11
\long\def\@makecaption#1#2{%
    \vskip 10pt
    \setbox\@tempboxa\hbox{%
\small{#1: }\ignorespaces #2}%
    \ifdim \wd\@tempboxa >\captionwidth {%
        \rightskip=\@captionmargin\leftskip=\@captionmargin
        \unhbox\@tempboxa\par}%
      \else
        \hbox to\hsize{\hfil\box\@tempboxa\hfil}%
    \fi}
\newdimen\@captionmargin\@captionmargin=2\parindent
\newdimen\captionwidth\captionwidth=\hsize
\catcode`\@=12
\def\makeautorefname#1#2{\expandafter\def\csname#1autorefname\endcsname{#2}}
\makeautorefname{lem}{Lemma}%
\makeautorefname{prop}{Proposition}%
\makeautorefname{rem}{Remark}%
\makeautorefname{section}{Section}%
\makeautorefname{subsection}{Section}%
\makeautorefname{subsubsection}{Section}%

\newcommand{\CupConnect}{\text{\----}}
\newcommand{\RayConnect}{\mid\hspace{-5pt}\text{\----}\,i}
\newcommand{\ba}{{\bf a}}

\DeclareMathOperator{\spn}{span}

\newcommand{\spr}{\mathcal{B}}
\newcommand{\espr}{\mathcal{B}^{\mathrm{e}}}
\newcommand{\dspr}{\mathcal{B}^{\Delta}}

\title{Two-row Delta Springer varieties}

\author{Abel Lacabanne}
\address{Laboratoire de Math{\'e}matiques Blaise Pascal (UMR 6620), Universit{\'e} Clermont Auvergne, Campus Universitaire des C{\'e}zeaux, 3 place Vasarely, 63178 Aubi{\`e}re Cedex, France\newline \href{http://www.normalesup.org/~lacabanne}{www.normalesup.org/$\sim$lacabanne}, \href{https://orcid.org/0000-0001-8691-3270}{ORCID 0000-0001-8691-3270}}
\email{abel.lacabanne@uca.fr}
\author{Pedro Vaz}
\address{Institut de Recherche en Math{\'e}matique et Physique, 
Universit{\'e} Catholique de Louvain, Chemin du Cyclotron 2,  
1348 Louvain-la-Neuve, Belgium, \newline \href{https://perso.uclouvain.be/pedro.vaz}{https://perso.uclouvain.be/pedro.vaz}, \href{https://orcid.org/0000-0001-9422-4707}{ORCID 0000-0001-9422-4707}}
\email{pedro.vaz@uclouvain.be}
\author{Arik Wilbert}
\address{Department of Mathematics \& Statistics,
University of South Alabama,
Mobile, AL 36688, USA
\newline \href{http://www.arik-wilbert.de/}{http://www.arik-wilbert.de/}, \href{https://orcid.org/0000-0003-3738-1428}{ORCID 0000-0003-3738-1428}
}
\email{wilbert@southalabama.edu}

\begin{document}
%
%
\newdimen\captionwidth\captionwidth=\hsize
%

\begin{abstract}
We study the geometry and topology of $\Delta$-Springer varieties associated with two-row partitions. These varieties were introduced in recent work by Griffin--Levinson--Woo to give a geometric realization of a symmetric function appearing in the Delta conjecture by Haglund--Remmel--Wilson. We provide an explicit and  combinatorial description of the irreducible components of the two-row $\Delta$-Springer variety and compare it to the ordinary two-row Springer fiber as well as Kato's exotic Springer fiber corresponding to a one-row bipartition. In addition to that, we extend the action of the symmetric group on the homology of the two-row $\Delta$-Springer variety to an action of a degenerate affine Hecke algebra and relate this action to a $\mathfrak{gl}_{2}$-tensor space.
\end{abstract}

\nnfootnote{\textit{Mathematics Subject Classification 2020.} Primary: 14M15; Secondary: 05E10, 20C08.}
\nnfootnote{\textit{Keywords.} Springer theory, flag varieties, action on homology, degenerate affine Hecke algebra.}

%
%
\maketitle

%

{\hypersetup{hidelinks}
\tableofcontents 
}
\pagestyle{myheadings}
\markboth{\em\small A.~Lacabanne, P.~Vaz and A.~Wilbert}{\em\small Two-row Delta Springer varieties}


\section{Introduction}

The Springer correspondence \cite{spr76, spr78} provides a powerful bridge between geometry (via Springer fibers) and algebra (via representations of Weyl groups) facilitating deep insights and results in both fields. In particular, it offers a geometric construction of the irreducible complex representations of Weyl groups. These representations are obtained in the top degree cohomology of Springer fibers, which are the fibers of the desingularization of the nilpotent cone.

In type $A$, nilpotent elements of the Lie algebra $\mathfrak{sl}_{n}$ are classified by their Jordan type, or equivalently, by a partition of $n$. Even in type $A$, the geometry of these fibers is not well understood. In the case of Springer fibers associated with two-row partitions, the situation is much nicer and has been studied extensively. For example, the irreducible components of two-row Springer fibers as well as their intersections are known to be smooth \cite{Fun03,SW12}. The Springer fibers for which all irreducible components are smooth have been classified in \cite{FrMe10}. In the two-row case, there also exists a nice diagrammatic description of the irreducible components in terms of cup diagrams \cite{Fun03,SW12}. A homeomorphic topological model has been built for these varieties \cite{Kho04,Weh09} and the action of the symmetric group on the top degree cohomology has a skein theoretic interpretation \cite{RuTy11,russell}. In types $C$ and $D$, the geometry and topology of two-row Springer fibers have been studied in \cite{EhSt16,Wil18,StWi19,ILW22}. As in type $A$, the irreducible components and their intersections are smooth and they admit explicit descriptions in terms of cup diagrams.

There also exists another version of the Springer correspondence for the symplectic group \cite{Kato06} which is cleaner than the original Springer correspondence in type $C$. In fact, Kato's exotic Springer correspondence yields a bijection between orbits in the exotic nilpotent cone and irreducible representations of the Weyl group of type $C$. In contrast to that, the original Springer correspondence is more intricate outside of type $A$ and requires extra data in terms of the component group. In the exotic case, the nilpotent orbits have been classified explicitly using bipartitions \cite{achar-henderson}. The irreducible components of exotic Springer fibers have been studied thoroughly in \cite{NaRoSa}, as well as~\cite{SW22} in the specific case of one-row bipartitions. As for the ordinary two-row Springer fibers, the irreducible components of exotic Springer fibers for one-row bipartitions can be described using certain cup diagrams.

Even though the study of Springer fibers originated in the geometric representation theory of Weyl groups, many connections to representation theory, combinatorics, geometry, and topology have been established in recent years. These connections are already rich and interesting when one restricts to the two-row case. For example, the diagrammatics appearing in the study of two-row Springer fibers have an interpretation in terms of parabolic Kazhdan--Lusztig theory \cite{Fun03,CDVDM,CDV}. Furthermore, the cohomology of two-row Springer fibers in type $A$ is related to Khovanov's arc algebra \cite{Kh-arc}, which provides invariants of tangles, and thus an interesting connection to low-dimensional topology. In fact, it turns out that the cohomology ring of the Springer fiber is isomorphic to the center of the principal block of parabolic category $\mathcal{O}$,~\cite{Br-cat-O,St-cat-O}. Using a generalization of Khovanov's arc algebras, deep connections to the representation theory of Lie (super)algebras and (walled) Brauer algebras were established in work by Brundan--Stroppel \cite{BS1,BS2,BS3,BS4} and Ehrig--Stroppel \cite{EhSt-diagrammatic,EhSt-koszul,EhSt17,EhSt}. As evident from the above, two-row Springer fibers have proven to have important applications in the field of categorification and $2$-representation theory.

\subsubsection*{Motivation}

The research that led to this paper originated in an attempt to understand \cite{LaNaVa} in terms of Springer theory, to define an arc algebra categorifying the Hecke algebra of type $B$ with unequal parameters, or more precisely one of its quotients, the blob algebra of Martin--Saleur \cite{MaSa-blob}. The representation theory of this algebra is governed by one-row bipartitions which naturally appear in Kato's exotic Springer correspondence. The exotic Springer fibers associated with one-row bipartitions share many geometric properties with the ordinary two-row Springer fibers in type $A$. Also note that the combinatorics of the blob algebra naturally appear when studying exotic Springer fibers for one-row bipartitions, \cite{SW22}. 

Recently, yet another Springer-type variety, called a $\Delta$-Springer variety, has been introduced in work of Griffin--Levinson--Woo,~\cite{GLW21}. This variety gives a geometric interpretation of a ring generalizing both the cohomology ring of a Springer fiber in type $A$, and the Haglund--Rhoades--Shimozono ring, \cite{HRS18}. As remarked in \cite{GiGr23}, the $\Delta$-Springer variety turns out to be a generalized Springer fiber in the sense of Borho--MacPherson, \cite{BoMc}. In the two-row case, the $\Delta$-Springer variety is intimately related to the exotic Springer fiber (see \autoref{main_thm_B}). We believe that it would be interesting to define an arc algebra whose center is isomorphic to the cohomology of a two-row $\Delta$-Springer variety, and establish connections with Kazhdan--Lusztig theory, generalizing the rich picture known in type $A$.

\subsubsection*{What we do in this paper}

In this paper, we specifically study $\Delta$-Springer varieties associated with two-row partitions. We develop a diagrammatic combinatorics well suited for comparison with ordinary two-row Springer fibers as well as exotic Springer fibers associated with one-row bipartitions. 

In order to define the $\Delta$-Springer variety in the two-row case, we fix a two-row partition $(n-k,k)$ of $n$ and an integer $0\leq m \leq k$. Moreover, we fix a nilpotent endomorphism $x$ of $\mathbb{C}^{n}$ of Jordan type $(n-k,k)$. The $\Delta$-Springer variety $\dspr_{(n-k,k),m}$ consists of all partial flags in $\mathbb{C}^n$ of the form $\{0\} = F_{0} \subseteq F_{1} \subseteq \dots \subseteq F_{n-m} \subseteq \mathbb{C}^n$ satisfying $\dim(F_{i})= i$, $xF_{i}\subseteq F_{i-1}$ for all $1\leq i \leq n-m$, and $\im x^{m}\subseteq F_{n-m}$. We refer to \autoref{def:DeltaSpr} for the general case.

The first step in our study is to provide an explicit description of the irreducible components of the two-row $\Delta$-Springer fibers using the notion of a $\Delta$-cup diagram. An example of a $\Delta$-cup diagram is given by
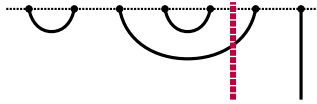
\begin{gather*}
\begin{tikzpicture}[scale=0.6]
\draw[thick,densely dotted] (-1.5,1) to (5.5,1);
\draw[black,fill=black] (-1,1) circle (2.0pt);
\draw[black,fill=black] (0,1) circle (2.0pt);
\draw[black,fill=black] (1,1) circle (2.0pt);
\draw[black,fill=black] (2,1) circle (2.0pt);
\draw[black,fill=black] (3,1) circle (2.0pt);
\draw[black,fill=black] (4,1) circle (2.0pt);
\draw[black,fill=black] (5,1) circle (2.0pt);
\draw[very thick] (-1,1) to[out=-80,in=180] (-.5,.5) to[out=0,in=-100] (0,1);
\draw[very thick] (1,1) to[out=-80,in=180] (2.5,-.1) to[out=0,in=-100] (4,1);
\draw[very thick] (2,1) to[out=-80,in=180] (2.5,.5) to[out=0,in=-100] (3,1);
\draw[very thick] (5,-1) to (5,1);
\draw[cutline] (3.5,-1) to (3.5,1.15);
\end{tikzpicture}
\end{gather*}
This is a crossingless diagram consisting of cups and rays attached to finitely many vertices on a horizontal line. In addition to that, there exists a vertical dotted red line, a so-called cut line, dividing the diagram into a left and a right part. We only allow right endpoints of cups and rays to the right of the cut line. More details on $\Delta$-cup diagrams can be found in \autoref{def:DeltaCup}. 

\autoref{sec:structure_irred_comp} is devoted to the study of the irreducible components of the $\Delta$-Springer variety $\dspr_{(n-k,k),m}$, and we prove the following result.

\begin{mainthm}[\autoref{thm:algebraic_main_result}]
    There exists a bijection between the irreducible components of the $\Delta$-Springer variety $\dspr_{(n-k,k),m}$ and the $\Delta$-cup diagrams on $n$ points with $k$ cups and $m$ vertices to the right of the cut line. We give explicit relations describing all flags contained in the irreducible component associated with a given $\Delta$-cup diagram. As a consequence, we show that each irreducible component is an iterated $\mathbb{P}^{1}$-bundle, and, in particular, it is smooth.
\end{mainthm}

To some extent, we like to think of the $\Delta$-Springer varieties as an interpolation between Springer fibers in type $A$ and exotic Springer fibers that have a type $C$ flavor. Indeed, in the extremal case $m=0$, the $\Delta$-Springer variety $\dspr_{(n-k,k),0}$ is equal to the two-row Springer fiber associated with the partition $(n-k,k)$, and in the extremal case $m=k$, the $\Delta$-Springer variety $\dspr_{(n-k,k),k}$ is isomorphic to the exotic Springer fiber associated with the bipartition $((n-2k),(k))$. Using our explicit description of irreducible components, we relate in \autoref{sec:comparison}, for any value of $0\leq m\leq k$, the two-row $\Delta$-Springer variety $\dspr_{(n-k,k),m}$ to the exotic Springer fiber $\espr_{((n-m-k),(k))}$ associated with a one-row bipartition. 

\begin{mainthm}[\autoref{prop:comparison_exotic_delta}] \label{main_thm_B}
    There exists an isomorphism of algebraic varieties from the  $\Delta$-Springer variety $\dspr_{(n-k,k),m}$ to a closed subvariety of the exotic Springer fiber $\espr_{((n-m-k),(k))}$.
\end{mainthm}

We also give a negative answer to \cite[Question 8.7]{GLW21} (see \autoref{ex:counter-ex}). The natural birational map from a union of irreducible components of the Springer fiber $\spr_{(n-k,k)}$ to the $\Delta$-Springer variety $\dspr_{(n-k,k),m}$ described in \cite[Remark 5.12]{GLW21} is not an isomorphism since it is not bijective.

\medskip

The next step is a representation theoretic study of the homology of the two-row $\Delta$-Springer variety. In \cite{GLW21}, an action of a symmetric group is constructed on each degree of the homology of any $\Delta$-Springer variety. Moreover, the top degree representation is identified with a Specht module associated with a skew partition. In \autoref{sec:top-model-action-symmetric}, we construct a topological model for the $\Delta$-Springer variety in the two-row case and use this model to identify the representation of the symmetric group $S_{n-m}$ in every degree of its homology (not only the top degree).

\begin{mainthm}[\autoref{prop:action_symmetric_hom}]
    For $0\leq d \leq k$, the degree $2d$ of the homology of the $\Delta$-Springer variety $\dspr_{(n-k,k),m}$ is isomorphic, as an $S_{n-m}$-representation, to the Specht module associated with the skew partition $(n-d,d)/(m)$.
\end{mainthm}

\begin{rem*}
    We construct the action of the symmetric group on the homology of the $\Delta$-Springer variety (or, more precisely, on its topology model) by embedding the homology into the homology of a product of $2$-spheres. The symmetric group naturally acts on the product of $2$-spheres by permuting spheres. We show that the action induced in homology restricts to an action on the homology of the $\Delta$-Springer variety. This construction also provides a skein theoretic description of this representation, answering \cite[Question 8.6]{GLW21}.
\end{rem*}

The above result shows that each degree of the homology of the $\Delta$-Springer variety is not always irreducible as a representation of the symmetric group $S_{n-m}$. In contrast to that, the representations of the symmetric group on the homology of ordinary two-row Springer fibers, as well as the representations of the Weyl group of type $C$ on the homology of exotic Springer fibers associated with one-row bipartitions, are irreducible. In \autoref{sec:degenerate-affine}, we extend the action of the symmetric group on the homology of the $\Delta$-Springer variety to an action of the degenerate affine Hecke algebra.

\begin{mainthm}[\autoref{thm:stability_degenerate} and \autoref{cor:irred_H_action_Delta}]
    Each degree of the homology of a two-row $\Delta$-Springer fiber is an irreducible representation of the degenerate affine Hecke algebra.
\end{mainthm}

In addition to the irreducibility of the representations in each degree, we find it surprising that the action of the degenerate affine Hecke algebra preserves the homological degree in the first place. In order to prove the above theorem, we identify the action on homology with the action of the degenerate affine Hecke algebra on a $\mathfrak{gl}_{2}$-tensor space, using a version of Schur--Weyl duality.

\subsubsection*{Conventions}

In this paper, all varieties and vector spaces are defined over the field of complex numbers. If $X$ is a topological space, we denote by $H_{*}(X)$ its singular homology with complex coefficients and by $H^{*}(X)$ its cohomology with complex coefficients. It follows from the universal coefficient theorem that homology and cohomology are dual to each other degreewise, that is $H_{i}(X) \cong \mathrm{Hom}(H^{i}(X),\mathbb{C})\cong H^{i}(X)$ for all nonnegative integers $i$. Note that this duality is different from Poincar\'e duality. The duality implies that all results originally proved for cohomology remain true for homology, and vice versa, as long as they only depend on the vector space structure.

\subsubsection*{Acknowledgements}

The authors would like to thank Sean T. Griffin for helpful conversations and clarifications on \cite{GLW21}. They also want to thank Alexis Langlois-Rémillard, Grégoire Naisse, Loïc Poulain d'Andecy, Simon Riche, Toshiaki Shoji, and Catharina Stroppel for discussions around this work, and the anonymous referees for their detailed comments. 

AL was partially supported by a PEPS JCJC grant from INSMI (CNRS). PV was supported by the Fonds de la Recherche Scientifique - FNRS under Grants no. MIS-F.4536.19 and CDR-J.0189.23. AW was supported by an AMS--Simons PUI Grant.

\section{\texorpdfstring{$\Delta$}{Delta}-Springer varieties}

Given an integer $n$ and a sequence $(d_1,\ldots,d_r)$, we will denote by $\mathcal{F}l_{(d_1,\ldots,d_r)}(\mathbb{C}^n)$ the set of partial flags $F_{\bullet} = (F_i)_{1 \leq i \leq r}$ such that $\dim(F_i/F_{i-1})=d_i$. Concerning flags, we will always use the convention that $F_0 = \{0\}$. We will write shortly $\mathcal{F}l(\mathbb{C}^n)$ for the set of complete flags in $\mathbb{C}^n$, that is the set $\mathcal{F}l_{(1^n)}(\mathbb{C}^n)$.

\subsection{Springer fibers}

Springer fibers arise as fibers of the desingularization of the nilpotent cone (see \cite{spr78}). The \emph{Springer fiber} $\spr_{x}$ associated with a nilpotent element $x\in \mathfrak{gl}_n(\mathbb{C})$ is the following subset of complete flags in $\mathbb{C}^n$:
\begin{gather*}
    \spr_x = \{ F_{\bullet} \in \mathcal{F}l(\mathbb{C}^n) \mid xF_i \subseteq F_{i-1} \text{ for } i\leq n\}.
\end{gather*}

This is a subvariety of the flag variety which, up to isomorphism, depends only on the orbit of $x$ under the action of $GL_n(\mathbb{C})$ by conjugation. Given a partition $\lambda$ of $n$, we will then usually write $\spr_{\lambda}$ for the Springer fiber associated with a nilpotent element in $\mathfrak{gl}_n(\mathbb{C})$ of Jordan type $\lambda$.

The study of the cohomology of these Springer fibers is related to the representation theory of the symmetric group via the Springer correspondence. Given a partition $\lambda = (\lambda_1,\ldots,\lambda_r)$ of $n$, we denote by $n(\lambda)$ the integer $\sum_{i=1}^r \frac{\lambda_i(\lambda_i-1)}{2}$.

\begin{thm}[{\cite[Section 1]{spr78}}]
    The following hold:
    \begin{enumerate}
        \item The Springer fiber $\spr_{\lambda}$ is equidimensional of dimension $n(\lambda)$. 

        \item There exists an irreducible action of the symmetric group $S_n$ on $H^{2n(\lambda)}(\spr_{\lambda})$.
        
        \item The map $\lambda \mapsto H^{2n(\lambda)}(\spr_{\lambda})$ is a bijection between partitions of $n$ and irreducible complex representations of $S_n$.
    \end{enumerate}    
\end{thm}

\subsection{Exotic Springer fibers}

Exotic Springer fibers arise as fibers of the desingularization of the exotic nilpotent cone (see~\cite{Kato06}). We first need to endow $\mathbb{C}^{2n}$ with a structure of a symplectic space. Fix a basis $(e_1,\ldots,e_n,f_1,\ldots,f_n)$ of $\mathbb{C}^{2n}$ and define a symplectic form $\omega$ by $\omega(e_i,f_j) = -\omega(f_j,e_i) = \delta_{i+j,n+1}$. The action of $\mathrm{Sp}(\mathbb{C}^{2n},\omega)$ on $\mathfrak{gl}_{2n}(\mathbb{C})$ by conjugation yields a decomposition $\mathfrak{gl}_{2n}(\mathbb{C}) = \mathfrak{sp}_{2n}(\mathbb{C})\oplus \mathcal{S}(\mathbb{C}^{2n})$ of $\mathrm{Sp}(\mathbb{C}^{2n},\omega)$-modules. Explicitly
\[
\mathcal{S}(\mathbb{C}^{2n}) = \{ x\in \mathfrak{gl}_{2n} \mid \forall v,w\in\mathbb{C}^{2n},\ \omega(xv,w)=\omega(v,xw)\}.
\]
The exotic Springer fiber $\espr_{x,v}$ associated with a nilpotent element $x\in \mathcal{S}(\mathbb{C}^{2n})$ and a vector $v\in\mathbb{C}^{2n}$ is the following subset of flags in $\mathbb{C}^{2n}$:
\begin{gather*}
    \espr_{x,v} = \{ F_{\bullet}\in\mathcal{F}l_{(1^n,n)}(\mathbb{C}^{2n}) \mid F_i\text{ is isotropic with respect to }\omega,\ xF_i\subseteq F_{i-1},\ v\in F_n\}. 
\end{gather*}

Once again, up to isomorphism, this algebraic variety only depends on the orbit of the pair $(x,v)$ under the action of $\mathrm{Sp}(\mathbb{C}^{2n},\omega)$. These orbits were determined in \cite{achar-henderson} and are indexed by bipartitions of $n$. Therefore, given a bipartition $(\lambda,\mu)$ of $n$, we will write $\espr_{\lambda,\mu}$ for the exotic Springer fiber instead of $\espr_{x,v}$ if the pair $(x,v)$ is in the orbit labeled by $(\lambda,\mu)$.

Similarly to the Springer correspondence in type $A$, we obtain a geometric construction of the irreducible complex representations of the Weyl group of type $C$.

\begin{thm}[\cite{Kato06}]
    The following hold:
    \begin{enumerate}
        \item The exotic Springer fiber $\espr_{\lambda,\mu}$ is equidimensional of dimension $n(\lambda)+\lvert\mu\rvert$. 

        \item There exists an irreducible action of the Weyl group of type $C_n$ on $H^{2(n(\lambda)+\lvert\mu\rvert)}(\espr_{\lambda,\mu})$.
        
        \item The map $(\lambda,\mu) \mapsto H^{2(n(\lambda)+\lvert\mu\rvert)}(\espr_{\lambda,\mu})$ is a bijection between bipartitions of $n$ and irreducible complex representations of the Weyl group of type $C_n$.
    \end{enumerate}
\end{thm}

\subsection{Definition of the \texorpdfstring{$\Delta$}{Delta}-Springer variety and basic results}\

Let $\lambda=(\lambda_1,\ldots,\lambda_s)$ be a partition of $n$, and let $m$ be a positive integer such that $0 \leq m \leq \lambda_s$. In terms of the Young diagram, we visualize $m$ as a cut line:
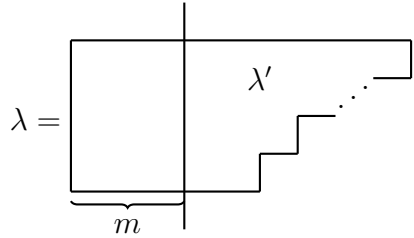
\[
\lambda =
\xy (0,0)*{
\tikzdiagc[scale=.5]{
\draw[thick] (-2,-2) -- (-2,2);
\draw[thick] (-2,2) -- (7,2);
\draw[thick] (7,2) -- (7,1);
\draw[thick] (7,1) -- (6,1);
\draw[thick] (-2,-2) -- (3,-2);
\draw[thick] (3,-2) -- (3,-1);
\draw[thick] (3,-1) -- (4,-1);
\draw[thick] (4,-1) -- (4,0);
\draw[thick] (4,0) -- (5,0);
\node at (5.45,.75) {$\iddots$};
\draw[thick] (1,-3) -- (1,3);
\draw[thick,decoration={brace, mirror, raise=.3em},decorate] (-2,-2) -- (1,-2) node [pos=0.5,anchor=north,yshift=-.5em] {$m$};
\node at (3,1) {$\lambda'$};
}}\endxy
\]
The part of the partition to the right of the cut line is denoted by $\lambda'$. Then $\lambda'=(\lambda_1-m,\ldots,\lambda_s-m)$ is a partition of $n'=n-ms$.

\begin{defn}\label{def:DeltaSpr}
Let $x\in \mathfrak{gl}_n(\mathbb{C})$ of Jordan type $\lambda$. We define the \emph{$\Delta$-Springer variety} by
\[
\dspr_{\lambda,m}=\{F_\bullet \in \mathcal Fl_{(1^{n'+m},m(s-1))}(\mathbb C^n)\mid xF_i \subseteq F_{i-1} \text{ for }i \leq n'+m\text{, and }\im x^m \subseteq F_{n'+m}\}.
\]
\end{defn}

Up to isomorphism, the algebraic variety $\dspr_{\lambda,m}$ depends only on the Jordan type of the nilpotent element $x$, see~\cite[Lemma 3.4]{GLW21}. Note that we have $\dspr_{\lambda,0}=\spr_{\lambda}$.

\begin{rem}
    The variety $\dspr_{\lambda,m}$ is the variety denoted by $Y_{n' + m, \lambda',s}$ in~\cite{GLW21}. This change of notation is justified by our comparison between Springer fibers, exotic Springer fibers and $\Delta$-Springer varieties when $\lambda$ is a two-row partition, see \autoref{sec:comparison}.
\end{rem}

Let $\mathcal T_{\lambda,m}$ be the set of partial fillings of the Young diagram of $\lambda$ with the labels $\{1,\ldots,n'+m\}$ (without repetition), such that the labels in each row are right justified and decrease from left to right, and the $i$th row contains at least $\lambda'_i$-many labels.

Let $\mathcal P(m,\lambda')$ be the set of all fillings of the Young diagram of $\lambda'$ with the labels $\{1,\ldots,n'+m\}$, such that the labels decrease from left to right along each row and down each column. Since $\lambda'$ is a partition of $n'$, we do not use all the possible labels in such a filling.

The following proposition summarizes some of the results from~\cite{GLW21}. 

\begin{thm} \label{thm:results_from_GLW}
The following hold. 
\begin{enumerate}
\item There exists an affine paving of $\dspr_{\lambda,m}$ whose cells are in bijection with the set $\mathcal T_{\lambda,m}$.
\item If $m=\lambda_s$, then there is a bijection between the irreducible components of $\dspr_{\lambda,m}$ and the set $\mathcal P(m,\lambda')$. If $m<\lambda_s$, then there is a bijection between the irreducible components of $\dspr_{\lambda,m}$ and those elements $S\in\mathcal P(m,\lambda')$ which satisfy the following condition: let $i_S\in\{1,\ldots,n'+m\}$ be the smallest number which does not appear in the filling $S$, then the bottom row of $\lambda'$ is filled up by a subset of the numbers $\{1,\ldots,i_S-1\}$.
\item The variety $\dspr_{\lambda,m}$ is equidimensional and its dimension equals 
\[
n(\tilde{\lambda'})+m(s-1),
\] 
where $\widetilde{\lambda'}$ is the conjugate of the partition $\lambda'$. 
\end{enumerate}
\end{thm}

As for the usual and the exotic Springer fiber, there is an action of a Weyl group on the top degree cohomology of the $\Delta$-Springer variety. We refer to \cite{jp79} for the notion of the Specht module associated with a skew partition. The following is \cite{GLW21}.

\begin{thm} \label{thm:cohomology_results_from_GLW}
    There exists an action of the symmetric group $S_{n'+m}$ on $H^{2(n(\tilde{\lambda})+m(s-1))}(\dspr_{\lambda,m})$ which is isomorphic to the Specht module associated with the skew partition $\lambda / (m^{s-1})$.
\end{thm}

In contrast with the (exotic) Springer correspondence, the top degree cohomology is not an irreducible representation of the symmetric group $S_{n'+m}$.

\subsection{Special Case: two-row partitions}

In this subsection, we restrict ourselves to the case of $\Delta$-Springer varieties associated with two-row partitions. We fix $\lambda=(n-k,k)$ a two-row partition of $n$. In particular, we have $0\leq k \leq \lfloor n/2\rfloor$. 

\begin{defn}\label{def:DeltaCup}
Fix a horizontal line with $n=(n-m)+m$ vertices labeled by the numbers $1,\ldots,n$ in increasing order from left to right. A \emph{cup diagram} is obtained by either connecting two vertices by a \emph{cup}, or by attaching a vertical \emph{ray} to a given vertex. We require that the resulting diagram is crossingless and that every vertex is connected to exactly one endpoint of a cup or ray. We use the notation $i\CupConnect j$ to indicate that vertices $i<j$ are connected by a cup. Moreover, we write $\RayConnect$ if vertex $i$ is connected to a ray. 

The set of all cup diagrams on $n$ vertices with $k$ cups such that the $m$ rightmost vertices labeled by $\{n-m+1,\ldots,n\}$ are connected to rays or to right endpoints of cups only is denoted by $\mathbb{B}_{n-k,k,m}$. The diagrams in $\mathbb B_{n-k,k,m}$ are called \emph{$\Delta$-cup diagrams}.
\end{defn}

\begin{rem}
    \autoref{def:DeltaCup} makes sense for any nonnegative integers $n,k,m$ such that $0\leq k \leq \lfloor n/2 \rfloor$ and $0\leq m \leq n$, but the $\Delta$-Springer variety is only defined when $0\leq m \leq k$. We will need the more general diagrams in the proof of \autoref{prop:K_a_contained}.
\end{rem}

\begin{exe}
   Below we give examples of $\Delta$-cup diagrams for $n=7$, $m=2$. We use a vertical dashed line to indicate the $m$ rightmost vertices.  
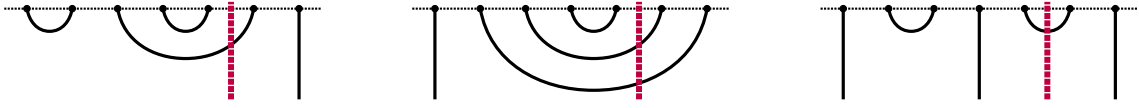
\[
\xy (0,0)*{
\begin{tikzpicture}[scale=0.6]
\draw[thick,densely dotted] (-1.5,1) to (5.5,1);
\draw[black,fill=black] (-1,1) circle (2.0pt);
\draw[black,fill=black] (0,1) circle (2.0pt);
\draw[black,fill=black] (1,1) circle (2.0pt);
\draw[black,fill=black] (2,1) circle (2.0pt);
\draw[black,fill=black] (3,1) circle (2.0pt);
\draw[black,fill=black] (4,1) circle (2.0pt);
\draw[black,fill=black] (5,1) circle (2.0pt);
\draw[very thick] (-1,1) to[out=-80,in=180] (-.5,.5) to[out=0,in=-100] (0,1);
\draw[very thick] (1,1) to[out=-80,in=180] (2.5,-.1) to[out=0,in=-100] (4,1);
\draw[very thick] (2,1) to[out=-80,in=180] (2.5,.5) to[out=0,in=-100] (3,1);
\draw[very thick] (5,-1) to (5,1);
\draw[cutline] (3.5,-1) to (3.5,1.15);
\end{tikzpicture}
}\endxy
\mspace{50mu}
\xy (0,0)*{
\begin{tikzpicture}[scale=0.6]
\draw[thick,densely dotted] (-1.5,1) to (5.5,1);
\draw[black,fill=black] (-1,1) circle (2.0pt);
\draw[black,fill=black] (0,1) circle (2.0pt);
\draw[black,fill=black] (1,1) circle (2.0pt);
\draw[black,fill=black] (2,1) circle (2.0pt);
\draw[black,fill=black] (3,1) circle (2.0pt);
\draw[black,fill=black] (4,1) circle (2.0pt);
\draw[black,fill=black] (5,1) circle (2.0pt);
\draw[very thick] (0,1) to[out=-80,in=180] (2.5,-.8) to[out=0,in=-100] (5,1);
\draw[very thick] (2,1) to[out=-80,in=180] (2.5,.5) to[out=0,in=-100] (3,1);
\draw[very thick] (1,1) to[out=-80,in=180] (2.5,-.1) to[out=0,in=-100] (4,1);
\draw[very thick] (-1,-1) to (-1,1);
\draw[cutline] (3.5,-1) to (3.5,1.15);
\end{tikzpicture}
}\endxy
\mspace{50mu}
\xy (0,0)*{
\begin{tikzpicture}[scale=0.6]
\draw[thick,densely dotted] (-1.5,1) to (5.5,1);
\draw[black,fill=black] (-1,1) circle (2.0pt);
\draw[black,fill=black] (0,1) circle (2.0pt);
\draw[black,fill=black] (1,1) circle (2.0pt);
\draw[black,fill=black] (2,1) circle (2.0pt);
\draw[black,fill=black] (3,1) circle (2.0pt);
\draw[black,fill=black] (4,1) circle (2.0pt);
\draw[black,fill=black] (5,1) circle (2.0pt);
\draw[very thick] (0,1) to[out=-80,in=180] (0.5,.5) to[out=0,in=-100] (1,1);
\draw[very thick] (3,1) to[out=-80,in=180] (3.5,.5) to[out=0,in=-100] (4,1);
\draw[very thick] (-1,-1) to (-1,1);
\draw[very thick] (2,-1) to (2,1);
\draw[very thick] (5,-1) to (5,1);
\draw[cutline] (3.5,-1) to (3.5,1.15);
\end{tikzpicture}
}\endxy
\]
The first two diagrams are elements of $\mathbb{B}_{4,3,2}$ and the last one an element of $\mathbb{B}_{5,2,2}$. Note that 
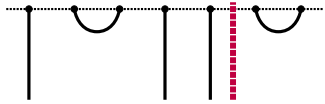
\[
\xy (0,0)*{
\begin{tikzpicture}[scale=0.6]
\draw[thick,densely dotted] (-1.5,1) to (5.5,1);
\draw[black,fill=black] (-1,1) circle (2.0pt);
\draw[black,fill=black] (0,1) circle (2.0pt);
\draw[black,fill=black] (1,1) circle (2.0pt);
\draw[black,fill=black] (2,1) circle (2.0pt);
\draw[black,fill=black] (3,1) circle (2.0pt);
\draw[black,fill=black] (4,1) circle (2.0pt);
\draw[black,fill=black] (5,1) circle (2.0pt);
\draw[very thick] (0,1) to[out=-80,in=180] (0.5,.5) to[out=0,in=-100] (1,1);
\draw[very thick] (4,1) to[out=-80,in=180] (4.5,.5) to[out=0,in=-100] (5,1);
\draw[very thick] (-1,-1) to (-1,1);
\draw[very thick] (2,-1) to (2,1);
\draw[very thick] (3,-1) to (3,1);
\draw[cutline] (3.5,-1) to (3.5,1.15);
\end{tikzpicture}
}\endxy
\]
is not a $\Delta$-cup diagram (the second rightmost vertex is not connected to a ray nor is it the right endpoint of a cup).
\end{exe}

\begin{lem}\label{lem:bijection_GLW_cups}
Let $0\leq m \leq k$. There exists a bijection between the irreducible components of $\dspr_{(n-k,k),m}$ and the set $\mathbb{B}_{n-k,k,m}$.   
\end{lem}

\begin{proof}
The irreducible components of $\dspr_{(n-k,k),m}$ are indexed by a subset of $\cP(m,\lambda')$ by \autoref{thm:results_from_GLW}. We remark that this subset of $\cP(m,\lambda')$ is in bijection with the set of fillings of the skew Young diagram of $\lambda/(m)$ with the labels in $\{1,\ldots,n-m\}$. Such a bijection is obtained by first keeping the labels of $\lambda'$ where they are and then completing the bottom row of $\lambda/(m)$ with the entries that are not in the filling of $\lambda'$.

We can now embed these fillings of $\lambda/(m)$ with the labels in $\{1,\ldots,n-m\}$ in the set of fillings of $\lambda$ with the labels in $\{1,\ldots,n\}$ by simply adding the entries $n,\ldots,n-m+1$ into the first $m$ boxes of the first row. The image of this embedding is exactly the set of fillings of $\lambda$ with the labels in $\{1,\ldots,n\}$ which are decreasing along rows and columns and such that the entries $n,\ldots,n-m+1$ are in the first row (the decreasing condition forces these entries to be in the first $m$ boxes).

Finally, to such a filling of $\lambda$ we associate the unique element of $\mathbb{B}_{n-k,k,m}$ such that the vertices connected to the left endpoints of cups are the entries of the second row. Thus we obtain a bijection between the irreducible components of $\dspr_{(n-k,k),m}$ and the set $\mathbb{B}_{n-k,k,m}$.
\end{proof}

\begin{exe}\label{ex:first_example}
Let us take $\lambda=(3,3)$ and $m=2$. The set of irreducible components is then parametrized by the following fillings of $\lambda'=(1,1)$ with entries in $\{1,2,3,4\}$
\[
\begin{ytableau}2\\1\end{ytableau},\begin{ytableau}3\\1\end{ytableau},\begin{ytableau}4\\1\end{ytableau}.
\]
The skew tableaux of shape $\lambda/(m)$ and the tableaux of shape $\lambda$ as obtained in the proof are
\[
\begin{ytableau}\none&\none&2\\4&3&1\end{ytableau},\begin{ytableau}\none&\none&3\\4&2&1\end{ytableau},\begin{ytableau}\none&\none&4\\3&2&1\end{ytableau}
\]
and
\[
\begin{ytableau}6&5&2\\4&3&1\end{ytableau},\begin{ytableau}6&5&3\\4&2&1\end{ytableau},\begin{ytableau}6&5&4\\3&2&1\end{ytableau}.
\]
Finally, the corresponding elements of $\mathbb{B}_{3,3,2}$ are 
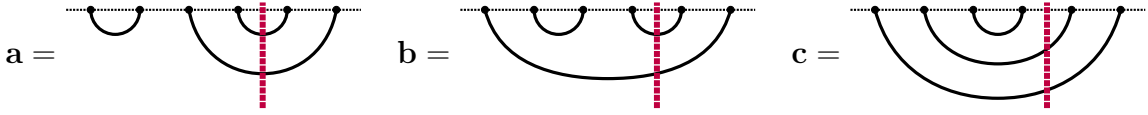
\begin{equation*}
\ba = \xy (0,0)*{
\begin{tikzpicture}[scale=0.65]
\draw[thick,densely dotted] (-.5,1) to (5.5,1);
\draw[black,fill=black] (0,1) circle (2.0pt);
\draw[black,fill=black] (1,1) circle (2.0pt);
\draw[black,fill=black] (2,1) circle (2.0pt);
\draw[black,fill=black] (3,1) circle (2.0pt);
\draw[black,fill=black] (4,1) circle (2.0pt);
\draw[black,fill=black] (5,1) circle (2.0pt);
\draw[very thick] (0,1) to[out=-90,in=180] (.5,.5) to[out=0,in=-90] (1,1);
\draw[very thick] (2,1) to[out=-80,in=180] (3.5,-.3) to[out=0,in=-100] (5,1);
\draw[very thick] (3,1) to[out=-90,in=180] (3.5,.5) to[out=0,in=-90] (4,1);
\draw[cutline] (3.5,-1) to (3.5,1.15);
\end{tikzpicture}
}\endxy
\mspace{20mu}
\mathbf{b} =\xy (0,0)*{
\begin{tikzpicture}[scale=0.65]
\draw[thick,densely dotted] (-.5,1) to (5.5,1);
\draw[black,fill=black] (0,1) circle (2.0pt);
\draw[black,fill=black] (1,1) circle (2.0pt);
\draw[black,fill=black] (2,1) circle (2.0pt);
\draw[black,fill=black] (3,1) circle (2.0pt);
\draw[black,fill=black] (4,1) circle (2.0pt);
\draw[black,fill=black] (5,1) circle (2.0pt);
\draw[very thick] (1,1) to[out=-90,in=180] (1.5,.5) to[out=0,in=-90] (2,1);
\draw[very thick] (0,1) to[out=-75,in=180] (2.5,-.4) to[out=0,in=-105] (5,1);
\draw[very thick] (3,1) to[out=-90,in=180] (3.5,.5) to[out=0,in=-90] (4,1);
\draw[cutline] (3.5,-1) to (3.5,1.15);
\end{tikzpicture}
}\endxy
\mspace{20mu}
\mathbf{c} =\xy (0,0)*{
\begin{tikzpicture}[scale=0.65]
\draw[thick,densely dotted] (-.5,1) to (5.5,1);
\draw[black,fill=black] (0,1) circle (2.0pt);
\draw[black,fill=black] (1,1) circle (2.0pt);
\draw[black,fill=black] (2,1) circle (2.0pt);
\draw[black,fill=black] (3,1) circle (2.0pt);
\draw[black,fill=black] (4,1) circle (2.0pt);
\draw[black,fill=black] (5,1) circle (2.0pt);
\draw[very thick] (1,1) to[out=-80,in=180] (2.5,-.1) to[out=0,in=-100] (4,1);
\draw[very thick] (0,1) to[out=-75,in=180] (2.5,-.8) to[out=0,in=-105] (5,1);
\draw[very thick] (2,1) to[out=-90,in=180] (2.5,.5) to[out=0,in=-90] (3,1);
\draw[cutline] (3.5,-1) to (3.5,1.15);
\end{tikzpicture}
}\endxy
\end{equation*}
\end{exe}

\section{Irreducible components for two-row \texorpdfstring{$\Delta$}{Delta}-Springer varieties} \label{sec:structure_irred_comp}

As in the previous section, let $\lambda=(n-k,k)$ be a two-row partition of $n$ and $0 \leq m \leq k$. We also define $\lambda'=(n-k-m,k-m)$ which is a two-row partition of $n'=n-2m$.

\subsection{Embedding the \texorpdfstring{$\Delta$}{Delta}-Springer variety into the Cautis--Kamnitzer variety} \label{sec:embedded_Springer_fiber}

Fix a large integer $N>0$ (see \autoref{y_remark} for details on what is considered ``large'') and let $z\colon\mathbb C^{2N}\to\mathbb C^{2N}$ be a nilpotent linear endomorphism with two Jordan blocks of the same size. In particular, there exists a Jordan basis 
\begin{equation} \label{eq:Jordan_basis_of_z}
\begin{tikzpicture}[baseline={(0,0)}]
\node (e1) at (0,0) {$e_1$};
\node (e2) at (1,0) {$e_2$};
\node (space1) at (2,0) {$\ldots{}^{}$};
\node (eN) at (3.1,0) {$e_N$};
\node (f1) at (5,0) {$f_1$};
\node (f2) at (6,0) {$f_2$};
\node (space2) at (7,0) {$\ldots{}^{}$};
\node (fN) at (8.1,0) {$f_N$};
\path[->,font=\scriptsize,>=angle 90,bend right]
(e2) edge (e1)
(space1) edge (e2)
(eN) edge (space1)
(f2) edge (f1)
(space2) edge (f2)
(fN) edge (space2);
\end{tikzpicture}
\end{equation}
of $\mathbb C^{2N}$ where the action of $z$ is indicated by the arrows (the vectors $e_1$ and $f_1$ are sent to zero). 

In~\cite[\S2]{CK08}, Cautis--Kamnitzer define a smooth, projective variety given by
\begin{equation} \label{eq:Y_i}
Y_{n}:=\{F_{\bullet}\in \mathcal{F}l_{(1^n,2N-n)}(\mathbb{C}^{2N})\mid zF_{i}\subseteq F_{i-1}\},
\end{equation}
which will play an important role for our results.

\begin{rem} \label{y_remark}
Note that the inclusions $zF_i\subseteq F_{i-1}$ imply that
\begin{displaymath}
F_n \subseteq z^{-1}F_{n-1} \subseteq\ldots\subseteq z^{-n}(0) = \spn(e_1,\ldots,e_n,f_1,\ldots,f_n). 
\end{displaymath}
Hence, the variety $Y_n$ does not depend on the choice of $N$ as long as $N \geq n$. In particular, we can always assume (by making $N$ larger, if necessary) that all the vector spaces of a flag in $Y_n$ are contained in the image of $z$. 
\end{rem}

Define $E_{n-k,k}\subseteq\mathbb C^{2N}$ to be the subspace spanned by 
\[
e_1,\ldots,e_{n-k},f_1,\ldots,f_{k}.
\]

Then we can view the $\Delta$-Springer variety $\dspr_{(n-k,k),m}$ as a subvariety of $Y_{n-m}$ via the following identification 
\begin{equation} \label{eq:delta_Y}
\dspr_{(n-k,k),m} \cong \left\{ F_{\bullet} \in Y_{n-m} \mid z^{m}\left(E_{n-k,k}\right)\subseteq F_{n-m}\subseteq E_{n-k,k}\right\}.
\end{equation} 
By the $z$-invariance of the flags, the following observation is immediate:

\begin{lem}\label{lem:vectors_last}
    We have $\dspr_{(n-k,k),m}\cong\left\{ F_{\bullet} \in Y_{n-m} \mid F_{n-m}\subseteq E_{n-k,k}, e_{n-k-m},f_{k-m}\in F_{n-m}\right\}$.
\end{lem}

\subsection{Explicit description of the irreducible components}

For the remainder of this article, we will write $\dspr_{(n-k,k),m}$ to denote the embedded $\Delta$-Springer variety via the identification~\eqref{eq:delta_Y}. The following subvarieties will describe irreducible components of the $\Delta$-Springer variety $\dspr_{(n-k,k),m}$.

\begin{defn}\label{def:Ka}
Let $\ba\in\mathbb B_{n-k,k,m}$. We define $K_\ba\subseteq Y_{n-m}$ to be the subvariety of $Y_{n-m}$ consisting of all flags $(F_1,\ldots,F_{n-m})$ satisfying the following conditions imposed by the $\Delta$-cup diagram $\ba$:
\begin{enumerate}[(i)]
\item If $i \CupConnect j$, $i,j\in\{1,\ldots,n-m\}$, then
\[
F_j=z^{-\frac{1}{2}(j-i+1)}F_{i-1}.
\]
\item If $\RayConnect$, $i\in\{1,\ldots,n-m\}$, then
\[
F_i=F_{i-1}+\spn\Bigl(e_{\frac{1}{2}\bigl(i+\rho_\ba(i)\bigr)}\Bigr).
\]
Here, $\rho_\ba(i)$ is the number of rays to the left of vertex $i$ (including the vertex $i$) in $\ba$. 
\end{enumerate}
There are no relations for a vector space indexed by a vertex connected to a cup whose right endpoint is connected to a vertex in $\{n-m+1,\ldots,n\}$.
\end{defn}

\begin{thm}\label{thm:algebraic_main_result}
Let $0\leq m \leq k$. The following statements hold:
\begin{enumerate}[(a)]
\item\label{second_part_thm} The subvariety $K_\ba\subseteq Y_{n-m}$ is an irreducible component of the $\Delta$-Springer variety $\dspr_{(n-k,k),m}\subseteq Y_{n-m}$.
\item\label{third_part_thm} The irreducible component $K_\ba\subseteq \dspr_{(n-k,k),m}$ is a $k$-fold iterated fiber bundle over $\mathbb P^1$: there exist spaces $K_\ba=X_1,X_2,\ldots,X_{k},X_{k+1}=\mathrm{pt}$ together with maps $p_1,p_2,\ldots,p_{k}$ such that $p_j\colon X_j\to \mathbb P^1$ is a fiber bundle with typical fiber $X_{j+1}$. In particular, the irreducible component $K_\ba$ is smooth.
\item\label{first_part_thm} The map $\ba\mapsto K_\ba$ defines a bijection between the $\Delta$-cup diagrams in $\mathbb{B}_{n-k,k,m}$ and the irreducible components of $\dspr_{(n-k,k),m}$.
\end{enumerate}
\end{thm}

\begin{rem}
We could replace $z$ by the restriction $z_{\lambda}$ of $z$ to $E_{n-k,k}$ in \autoref{def:Ka}, which is justified by \autoref{prop:K_a_contained}. Hence, the description of the irreducible components of the $\Delta$-Springer variety in \autoref{thm:algebraic_main_result}\ref{second_part_thm} also makes sense without the embedding into $Y_{n-m}$.
\end{rem}

For the proof of \autoref{thm:algebraic_main_result} we consider the subvariety $X^i_{n-m} \subseteq Y_{n-m}$, $1\leq i < n-m$, defined by 
\begin{equation}\label{eq:X_i}
X^i_{n-m}:=\{F_{\bullet} \in Y_{n-m} \mid F_{i+1}=z^{-1}F_{i-1}\},
\end{equation}
and the surjective morphism of varieties $q_{n-m}^i \colon X^i_{n-m} \twoheadrightarrow Y_{n-m-2}$ given by
\begin{equation}\label{eq:q_i_morphism}
(F_1,\ldots,F_{n-m}) \mapsto \left(F_1,\ldots,F_{i-1},zF_{i+2},\ldots,zF_{n-m}\right),
\end{equation}
see also~\cite[\S2]{CK08}.

\begin{lem}\label{lem:induction_alg_components}
Let $\ba\in\mathbb{B}_{n-k,k,m}$ be a cup diagram with a cup connecting vertices $i$ and $i+1$ and let $\tilde{\ba}\in\mathbb{B}_{n-k-1,k-1,m}$ be the cup diagram obtained by deleting this cup. Then we have $K_\ba=(q_{n-m}^i)^{-1}(K_{\tilde{\ba}})$. 
\end{lem}
\begin{proof}
We have to show that a flag $(F_1,\ldots,F_{n-m})\in Y_{n-m}$ satisfies the conditions (i), (ii) from \autoref{def:Ka} with respect to the $\Delta$-cup diagram $\ba$ if and only if 
\[
q_{n-m}^i(F_1,\ldots,F_{n-m})=(F_1,\ldots,F_{i-1},zF_{i+2},\ldots,zF_{n-m})\in Y_{n-m-2}
\]
satisfies these conditions with respect to $\tilde{\ba}$. For a proof, we refer to~\cite[Lemma 17]{SW22}.
\end{proof}

\begin{prop}\label{prop:K_a_contained}
Let $\ba\in\mathbb{B}_{n-k,k,m}$ and $F_{\bullet}\in K_{\ba}$. Then $F_{n-m}\subseteq E_{n-k,k}$ and $e_{n-k-m},f_{k-m}$ belong to $F_{n-m}$. 
In particular, if $0\leq m \leq k$, we have $K_{\ba}\subseteq\dspr_{(n-k,k),m}$.
\end{prop}

\begin{proof}
We use the same argument as in the proof of \cite[Proposition 18]{SW22}: we proceed by induction on the number of cups in $\ba$ with both endpoints on the left of the cut line.

If there is no cup with both endpoints on the left of the cut line, then $m\geq k$ and
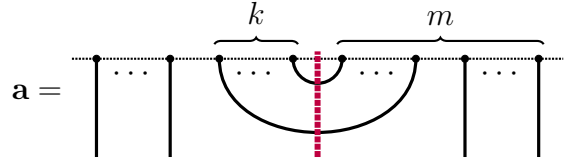
\begin{gather*}
    \ba =
\xy (0,2)*{
\begin{tikzpicture}[scale=0.65]
\draw[thick,densely dotted] (-.5,1) to (9.5,1);
\draw[black,fill=black] (0,1) circle (2.0pt);
\draw[black,fill=black] (1.5,1) circle (2.0pt);
\draw[black,fill=black] (2.5,1) circle (2.0pt);
\draw[black,fill=black] (4,1) circle (2.0pt);
\draw[black,fill=black] (5,1) circle (2.0pt);
\draw[black,fill=black] (6.5,1) circle (2.0pt);
\draw[black,fill=black] (7.5,1) circle (2.0pt);
\draw[black,fill=black] (9,1) circle (2.0pt);
\draw[very thick] (0,1) -- (0,-1);
\draw[very thick] (1.5,1) -- (1.5,-1);
\draw[very thick] (2.5,1) to[out=-90,in=180] (4.5,-.5) to[out=0,in=-90] (6.5,1);
\draw[very thick] (4,1) to[out=-90,in=180] (4.5,.5) to[out=0,in=-90] (5,1);
\draw[very thick] (7.5,1) -- (7.5,-1);
\draw[very thick] (9,1) -- (9,-1);
\draw[cutline] (4.5,-1) to (4.5,1.15);
\draw[thick,decoration={brace, raise=.3em},decorate] (2.4,1.1) -- (4.1,1.1) node [pos=0.5,anchor=north,yshift=2em] {$k$};
\draw[thick,decoration={brace, raise=.3em},decorate] (4.9,1.1) -- (9.1,1.1) node [pos=0.5,anchor=north,yshift=1.7em] {$m$};
\node at (.75,.7) {$\cdots$};
\node at (3.25,.7) {$\cdots$};
\node at (5.75,.7) {$\cdots$};
\node at (8.25,.7) {$\cdots$};
\end{tikzpicture}
}\endxy 
\end{gather*}
By convention, $f_{k-m}=0$ since $k\leq m$, so that $f_{k-m}\in F_{n-m}$. By definition of $K_{\ba}$, the flag $F_{\bullet}$ satisfies $F_i = \mathrm{span}(e_1,\ldots,e_i)$ for $1 \leq i \leq n-m-k$. Therefore, we have that $e_{n-k-m}\in F_{n-m}$ since $e_{n-k-m}\in F_{n-k-m}$ and $F_{n-k-m}\subseteq F_{n-m}$. Finally, $F_{n-m} \subseteq z^{-k}F_{n-k-m}$ is included in $\mathrm{span}(e_1,\ldots,e_{n-m},f_1,\ldots,f_k)=E_{n-k,k}$.

\medskip

Now, suppose that there is a cup in $\ba$ with both endpoints on the left of the cut line. Fix then $1 \leq i < n-m$ such that the vertices $i$ and $i+1$ are joined by a cup in $\ba$. We consider $\tilde{\ba}\in \mathbb{B}_{n-k-1,k-1,m}$ the diagram obtained by removing this cup.

Then $q_{n-m}^i(F_{\bullet}) = (F_1,\ldots,F_{i-1},zF_{i+2},\ldots,zF_{n-m})\in K_{\tilde{\ba}}$ and so, by induction, we have $zF_{n-m}{\subseteq} \mathrm{span}(e_1,\ldots,e_{n-k-1},f_1,\ldots,f_{k-1})$ and $F_{n-m}{\subseteq} \mathrm{span}(e_1,\ldots,e_{n-k},f_1,\ldots,f_k){=}E_{n-k,k}$.

We finally show that $e_{n-k-m}$ and $f_{k-m}$ are in $F_{n-m}$. By induction hypothesis, we also have $e_{n-k-m-1}\in zF_{n-m-1}$. There then exists $v\in F_{n-m-1}$ such that $z(v) = e_{n-k-m-1}$. Any such $v$ is of the form $v=e_{n-k-m}+\alpha e_1+ \beta f_1$. Since $i$ and $i+1$ are joined by a cup, we have $F_{i+1} = z^{-1}F_{i-1} \supseteq z^{-1}\{0\} = \mathrm{span}(e_1,f_1)$. Therefore the vectors $e_1$ and $f_1$ belong to $F_{n-m}$ and $e_{n-k-m} = v -\alpha e_1 - \beta f_1 \in F_{n-m}$. One shows similarly that $f_{k-m}\in F_{n-m}$.

\medskip

If $0\leq m\leq k$, then \autoref{lem:vectors_last} implies that $K_{\ba}\subset \dspr_{(n-k,k),m}$.
\end{proof}

\begin{proof}[Proof of \autoref{thm:algebraic_main_result}.]
We first note that $K_\ba$ is an $k$-fold iterated fiber bundle over $\mathbb P^1$. The proof is the same as for the irreducible components of two-row Springer fibers because the defining relations (i) and (ii) of $K_\ba$ in \autoref{def:Ka} are the same as for two-row Springer fibers. However, for the reader's convenience, we briefly recall the argument. We refer to~\cite[Proposition 5.1]{Fun03} and~\cite[Section 8]{Sch12} for additional details. Let $i_1<i_2<\ldots<i_{k}$ denote the vertices connected to a left endpoint of a cup in $\ba$. Note that the space $F_{i_1-1}$ is the same for every flag $(F_1,\ldots,F_{n-m})\in K_\ba$ because each vertex strictly to the left of $i_{1}$ is connected to a ray. Hence, by successively applying relation (ii) in \autoref{def:Ka}, we see that $F_{i_1-1}$ is uniquely determined. As a result, we can consider the fiber bundle
\[
p_1\colon K_\ba\to \mathbb P(z^{-1}F_{i_1-1}/F_{i_1-1})\cong\mathbb P^1\,,\,\,(F_1,\ldots,F_{n-m})\mapsto F_{i_1}/F_{i_1-1}.
\] 
Its typical fiber is denoted $X_2$ and consists of all flags $(F_1,\ldots,F_{n-m})\in K_\ba$ with $F_{i_1}$ (and $F_j$, if $(i_1+1)$ and $j$ are connected by a cup) fixed. Now we repeat the above construction replacing $X_1$ by $X_2$ and the vertex $i_1$ by $i_2$, and continue until we have exhausted all of the vertices $i_1<\cdots<i_{k}$.

In order to prove parts~\ref{second_part_thm} and~\ref{third_part_thm} of \autoref{thm:algebraic_main_result} we first note that $K_\ba$ is smooth since it is a $k$-fold iterated fiber bundle over $\mathbb P^1$. This shows that $K_\ba$ is irreducible since it is as well connected as an iterated fiber bundle over $\mathbb{P}^1$. By \autoref{prop:K_a_contained} the variety $K_\ba$ is contained in $\dspr_{(n-k,k),m}$. Finally, the dimension of $K_\ba$ equals $k$ which is the dimension of $\dspr_{(n-k,k),m}$ by \autoref{thm:results_from_GLW}. Hence, $K_\ba$ is an irreducible component of the (embedded) $\Delta$-Springer variety $\dspr_{(n-k,k),m}$. In particular, parts~\ref{second_part_thm} and~\ref{third_part_thm} of \autoref{thm:algebraic_main_result} are now clear from the above. 

By \autoref{lem:bijection_GLW_cups} we know that the cup diagrams in $\mathbb{B}_{n-k,k,m}$ are in bijective correspondence with the irreducible components of $\dspr_{(n-k,k),m}$. Since the irreducible components $K_\ba$ are different for different $\ba\in\mathbb B_{n-k,k,m}$, we see that the map $\ba\mapsto K_\ba$ explicitly realizes this bijection which proves part~\ref{first_part_thm} of the theorem.
\end{proof}

\begin{exe}
    Let us continue \autoref{ex:first_example}. The irreducible components are then given by
    \begin{itemize}
        \item $K_{\ba}=\{F_1 \subset \mathrm{span}(e_1,f_1) \subset F_2 \subset F_3 \}$,
        \item $K_{\mathbf{b}}=\{F_1 \subset F_2 \subset z^{-1}(F_1) \subset F_3\}$,
        \item $K_{\mathbf{c}}=\{F_1 \subset F_2 \subset F_3 \subset z^{-1}(F_2)\}$.
    \end{itemize}
\end{exe}

\subsection{\texorpdfstring{$\mathbb{C}^*$}{C*}-action and generalized components}

In \cite{GLW21}, an affine paving of the $\Delta$-Springer variety is constructed in order to compute the cohomology ring of the variety. We make this explicit in the two-row case using a $\mathbb{C}^*$-action.

Let $\mathbb{C}^*$ act on $\mathbb{C}^{2N}$ by
\[
    t\cdot e_i = t^{-1}e_i\quad\text{and}\quad t\cdot f_i = tf_i.
\]
This action restricts to $E_{n-k,k}$ and induces a $\mathbb{C}^*$-action on the $ \Delta$-Springer variety $\dspr_{(n-k,k),m}$.

\begin{defn}\label{def:combw}
    A \emph{combinatorial weight of type $(n-k,k)$} is a sequence in $\{\up,\down\}^{n}$ containing $n-k$ ups ($\up$) and $k$ downs ($\down$). 

    Such sequence is called a \emph{$\Delta$-weight of type $(n-k,k,m)$} if there is no $\down$ to the left of any $\up$ among the last $m$ symbols.
\end{defn}

When writing weights we use a $\vert$ to indicate the $m$ rightmost entries and we suppress the unnecessary commas.

\begin{exe}\label{ex:second_example}
    Let $n=5$ and $k=2$ and $m=2$. Then $\down\up\down\vert\!\up\up$ is a $\Delta$-weight of type $(3,2,2)$ and $\up\down\up\vert\!\down\up$ is not.
\end{exe}

\begin{prop}
    \label{prop:fixed_flags}
    There is a bijection between $\Delta$-weights of type $(n-k,k,m)$ and the fixed points under the action of $\mathbb{C}^*$ on $\dspr_{(n-k,k),m}$. 

    Explicitly, given a $\Delta$-weight $\alpha=(\alpha_1,\ldots,\alpha_n)$ of type $(n-k,k,m)$, the corresponding fixed point is the flag $F^{\alpha}_{\bullet}$ with $i$th subspace given by
    \begin{gather*}
        F^{\alpha}_i = \mathrm{span}(e_1,\ldots,e_{\#\{\up\text{'s weakly to the left of }i\}},f_1,\ldots,f_{\#\{\down\text{'s weakly to the left of }i\}}).
    \end{gather*}
\end{prop}

\begin{proof}
    It is clear that the flag $F^{\alpha}_{\bullet}$ is fixed by the $\mathbb{C}^*$-action since each space of the flag is generated by weight vectors with respect to $\mathbb{C}^*$-action.

    Conversely, let $F_{\bullet}=(F_1,\ldots,F_{n-m})\in\dspr_{(n-k,k),m}$ fixed under the action of $\mathbb{C}^*$. Using induction on $l$, one see that the vector space $F_l$ must be spanned by $e_1,\ldots,e_p,f_1,\ldots,f_q$ for some $0\leq p \leq n-k$ and $0 \leq q \leq k$ with $p+q=l$. By counting occurrences of $e$'s and $f$'s, we recover the entries $\alpha_1,\ldots,\alpha_{n-m}$ of the $\Delta$-weight corresponding to the fixed flag $F_{\bullet}$. Let $r$ be the number of $\down$ among these $\alpha_1,\ldots,\alpha_{n-m}$. Then the condition $z^m(E_{n-k,k})\subseteq F_{n-m}$ forces to have $0\leq r \leq k-m$. We then complete the weight $\alpha_1,\ldots,\alpha_{n-m}$ by $r+m-k$ $\up$ followed by $k-r$ $\down$ in order to obtain the $\Delta$-weight of type $(n-k,k,m)$ corresponding to $F_{\bullet}$.
\end{proof}

We define the attracting cell of the fixed point $F^{\alpha}_{\bullet}$ by
\begin{gather*}
    K_{\alpha} = \{F_{\bullet} \in \dspr_{(n-k,k),m} \mid \lim_{t\to +\infty} t\cdot F_{\bullet} = F^{\alpha}_{\bullet}\}.
\end{gather*}
To test whether a flag $F_{\bullet}$ lies inside $K_{\alpha}$, we follow \cite[Section 2.2]{SW12}. We define $P$ to be the subspace of $E_{n-k,k}$ spanned by the $e_i$'s and $Q$ to be the subspace of $E_{n-k,k}$ spanned by the $f_i$'s. Given a flag $F_{\bullet}$ in $\dspr_{(n-k,k),m}$, we associate a new flag $F_{\bullet}^{\mathrm{ass}}$ by setting $F_i^{\mathrm{ass}} = P_i + Q_i \subset P\oplus Q$, where $P_i = F_i\cap P$ and $Q_i$ is the image of $F_i$ by the projection onto $Q$ along $P$. It is then clear that this flag is stable under the $\mathbb{C}^*$-action. The following is similar to \cite[Proposition 14]{SW12}.

\begin{lem}
    \label{lem:associated_flag} 
    Let $\alpha$ be a $\Delta$-weight of type $(n-k,k,m)$ and $F_{\bullet}\in \dspr_{(n-k,k),m}$.
    The flag $F_{\bullet}$ is in the attracting cell $K_{\alpha}$ if and only $F_{\bullet}^{\mathrm{ass}} = F_{\bullet}^{\alpha}$.
\end{lem}

\autoref{lem:associated_flag} implies that the cell $K_{\alpha}$ is an affine variety. We can also explicitly describe these attracting cells. Let $\alpha$ be a $\Delta$-weight and construct a $\Delta$-cup diagram $C(\alpha)$ as follows. First successively connect neighboring pairs $\down\up$ by a cup, ignoring symbols that are already connected. When there are no more neighboring pairs $\down\up$ among the remaining symbols, then connect all remaining symbols to ray.

\begin{thm}\label{thm:paving}~
    \begin{enumerate}
    \item There is a bijection between the $\Delta$-weights of type $(n-k,k,m)$ and the cells of an affine paving of $\dspr_{(n-k,k),m}$.
    \item The attracting cell $K_{\alpha}$ consists of all flags $F_{\bullet}\in\dspr_{(n-k,k),m}$ satisfying the following conditions:
    \begin{enumerate}[(i)]
        \item $F_j = z^{-\frac{j-i+1}{2}}F_{i-1}$ if $1 \leq i < j \leq n-m$ and $i \CupConnect j$ in $C(\alpha)$,
        \item $F_i=F^{\alpha}_i$ if $\RayConnect$ in $C(\alpha)$,
        \item $F_{i-1} \cap P = F_i \cap P$ if $i$ is the left endpoint of a cup in $C(\alpha)$.
    \end{enumerate}
    \end{enumerate}
\end{thm}

\begin{proof}
    The first item follows from \autoref{prop:fixed_flags}. Concerning the second item, we can use the same arguments as in \cite[Theorem 36]{SW22}; working with a $\Delta$-Springer variety does not affect the proof.
\end{proof}

\begin{rem}
    We also obtain a description of the closures of the cells $K_{\alpha}$ by removing the condition (iii) in the description of $K_{\alpha}$ in the above theorem.

    In addition, if $C(\alpha)$ contains $k$ cups, then it is clear that the closure of $K_{\alpha}$ is the irreducible component associated with the $\Delta$-cup diagram $C(\alpha)$.
\end{rem}

\begin{exe}
\label{ex:weight-cup-diagrams}
    Let $\lambda=(3,3)$ and $m=2$.
    \begin{itemize}
        \item for $\alpha=\down\down\up\down\vert\!\up\up$ we get
        \begin{gather*}
            C(\alpha) = 
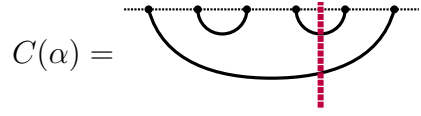
\xy (0,0)*{
            \begin{tikzpicture}[scale=0.65]
            \draw[thick,densely dotted] (-.5,1) to (5.5,1);
            \draw[black,fill=black] (0,1) circle (2.0pt);
            \draw[black,fill=black] (1,1) circle (2.0pt);
            \draw[black,fill=black] (2,1) circle (2.0pt);
            \draw[black,fill=black] (3,1) circle (2.0pt);
            \draw[black,fill=black] (4,1) circle (2.0pt);
            \draw[black,fill=black] (5,1) circle (2.0pt);
            \draw[very thick] (1,1) to[out=-90,in=180] (1.5,.5) to[out=0,in=-90] (2,1);
            \draw[very thick] (0,1) to[out=-75,in=180] (2.5,-.4) to[out=0,in=-105] (5,1);
            \draw[very thick] (3,1) to[out=-90,in=180] (3.5,.5) to[out=0,in=-90] (4,1);
            \draw[cutline] (3.5,-1) to (3.5,1.15);
            \end{tikzpicture}
            }\endxy
        \end{gather*}
        and $F_{\bullet}^{\alpha} = (\mathrm{span}(f_1),\mathrm{span}(f_1,f_2),\mathrm{span}(e_1,f_1,f_2),\mathrm{span}(e_1,f_1,f_2,f_3))$. Then the closure of the attracting cell consists of all flags $F_{\bullet}\in\dspr_{(3,3),2}$ such that
        \begin{gather*}
            F_1\subseteq F_2 \subseteq z^{-1}(F_1) \subseteq F_4 \subseteq \mathbb{C}^6.
        \end{gather*}
        This is in fact the irreducible component of $\dspr_{(3,3),2}$ labeled by the diagram $C(\alpha)$.
        \item for $\alpha = \up\down\down\down\vert\!\up\up$, we get
        \begin{gather*}
            C(\alpha)=
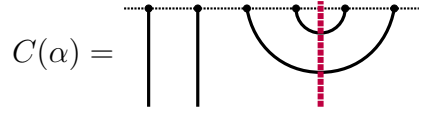
\xy (0,0)*{
            \begin{tikzpicture}[scale=0.65]
            \draw[thick,densely dotted] (-.5,1) to (5.5,1);
            \draw[black,fill=black] (0,1) circle (2.0pt);
            \draw[black,fill=black] (1,1) circle (2.0pt);
            \draw[black,fill=black] (2,1) circle (2.0pt);
            \draw[black,fill=black] (3,1) circle (2.0pt);
            \draw[black,fill=black] (4,1) circle (2.0pt);
            \draw[black,fill=black] (5,1) circle (2.0pt);
            \draw[very thick] (0,1) to (0,-1);
            \draw[very thick] (1,1) to (1,-1);
            \draw[very thick] (2,1) to[out=-80,in=180] (3.5,-.3) to[out=0,in=-100] (5,1);
            \draw[very thick] (3,1) to[out=-90,in=180] (3.5,.5) to[out=0,in=-90] (4,1);
            \draw[cutline] (3.5,-1) to (3.5,1.15);
            \end{tikzpicture}
            }\endxy
        \end{gather*}
        and $F_{\bullet}^{\alpha} = (\mathrm{span}(e_1),\mathrm{span}(e_1,f_1),\mathrm{span}(e_1,f_1,f_2),\mathrm{span}(e_1,f_1,f_2,f_3))$. Then the closure of the attracting cell consists of all flags $F_{\bullet}\in\dspr_{(3,3),2}$ such that
        \begin{gather*}
            \mathrm{span}(e_1)\subseteq \mathrm{span}(e_1,f_1) \subseteq F_3 \subseteq F_4 \subseteq \mathbb{C}^6.
        \end{gather*}
        This attracting cell is not an irreducible component of $\dspr_{(3,3),2}$.
        \item for $\alpha = \up\up\up\down\vert\!\down\down$, we get
        \begin{gather*}
            C(\alpha)=
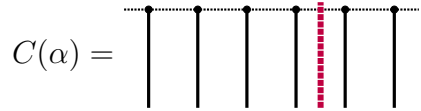
\xy (0,0)*{
            \begin{tikzpicture}[scale=0.65]
            \draw[thick,densely dotted] (-.5,1) to (5.5,1);
            \draw[black,fill=black] (0,1) circle (2.0pt);
            \draw[black,fill=black] (1,1) circle (2.0pt);
            \draw[black,fill=black] (2,1) circle (2.0pt);
            \draw[black,fill=black] (3,1) circle (2.0pt);
            \draw[black,fill=black] (4,1) circle (2.0pt);
            \draw[black,fill=black] (5,1) circle (2.0pt);
            \draw[very thick] (0,1) to (0,-1);
            \draw[very thick] (1,1) to (1,-1);
            \draw[very thick] (2,1) to (2,-1);
            \draw[very thick] (3,1) to (3,-1);
            \draw[very thick] (4,1) to (4,-1);
            \draw[very thick] (5,1) to (5,-1);
            \draw[cutline] (3.5,-1) to (3.5,1.15);
            \end{tikzpicture}
            }\endxy
        \end{gather*}
        and $F_{\bullet}^{\alpha} = (\mathrm{span}(e_1),\mathrm{span}(e_1,e_2),\mathrm{span}(e_1,e_2,e_3),\mathrm{span}(e_1,e_2,e_3,f_1))$. Then the closure of the attracting cell consists only of the flag $F^{\alpha}_{\bullet}$
    \end{itemize}
\end{exe}

As a corollary to \autoref{thm:paving}, we obtain a diagrammatic description of the homology of the $\Delta$-Springer variety.

\begin{cor}
    The homology $H_{\ast}(\dspr_{(n-k,k),m})$ has a basis indexed by the $\Delta$-weights of type $(n-k,k,m)$. Moreover, the homological degree of an element of this basis is given by twice the number of cups in $C(\alpha)$.
\end{cor}

\section{Comparison with Springer fibers and exotic Springer fibers}
\label{sec:comparison}

We still work with a $\Delta$-Springer variety $\dspr_{(n-k,k),m}$ for a two-row partition $(n-k,k)$ and $0\leq m \leq k$. We compare this variety with the two-row Springer fiber $\spr_{(n-k,k)}$ and with the exotic Springer fiber $\espr_{((n-m-k),(k))}$. 

\subsection{Comparison with the two-row Springer fiber}
\label{sec:delta-usual}

In this subsection, we give a negative answer to \cite[Question 8.7]{GLW21}. We already have remarked that if $m=0$ then the $\Delta$-Springer variety $\dspr_{\lambda,0}$ is equal to the Springer fiber $\spr_{\lambda}$. Consider $\pi \colon Y_n \rightarrow Y_{n-m}$ the morphism of algebraic varieties which forgets the last $m$ subspaces of a flag:
\begin{gather*}
    \pi(F_1,\ldots,F_n) = (F_1,\ldots,F_{n-m}).
\end{gather*}
Forgetting the cut line defines an injection $\iota\colon\mathbb{B}_{n-k,k,m}\rightarrow\mathbb{B}_{n-k,k,0}$. The set $\mathbb{B}_{n-k,k,m}$ indexes the irreducible components of the $\Delta$-Springer variety $\dspr_{(n-k,k),m}$ and the set $\mathbb{B}_{n-k,k,0}$ indexes the irreducible components of the Springer fiber $\spr_{(n-k,k)}$. Therefore, if $\ba\in\mathbb{B}_{n-k,k,m}$, then $K_{\iota(\ba)}$ is an irreducible component of $\spr_{(n-k,k)}$. In \cite[Remark 5.12]{GLW21}, the authors showed that $\pi$ induces a birational morphism form $\bigcup_{\ba\in \mathbb{B}_{n-k,k,m}}K_{\iota(\ba)}$ to $\dspr_{(n-k,k),m}$, and asked whether this map is an isomorphism. We now give an example that answers this question negatively.

\begin{exe}
    \label{ex:counter-ex}
    Let $\lambda=(3,2)$ and $m=2$. The corresponding $\Delta$-Springer variety $\dspr_{(3,2),2}$ is the union of three irreducible components. Consider the flag 
    \[
    F_{\bullet}=(F_1,F_2,F_3) = (\mathrm{span}(e_1),\mathrm{span}(e_1,e_2),\mathrm{span}(e_1,e_2,f_1)).
    \]
    This flag lies in the irreducible components corresponding to the following diagrams:
    
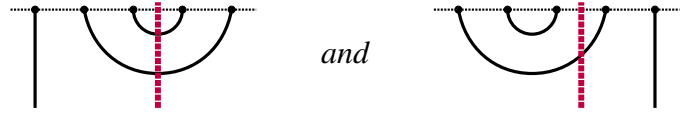
\begin{equation*}
    \xy (0,0)*{
    \begin{tikzpicture}[scale=0.65]
    \draw[thick,densely dotted] (-.5,1) to (4.5,1);
    \draw[black,fill=black] (0,1) circle (2.0pt);
    \draw[black,fill=black] (1,1) circle (2.0pt);
    \draw[black,fill=black] (2,1) circle (2.0pt);
    \draw[black,fill=black] (3,1) circle (2.0pt);
    \draw[black,fill=black] (4,1) circle (2.0pt);
    \draw[very thick] (0,1) to (0,-1);
    \draw[very thick] (1,1) to[out=-80,in=180] (2.5,-.3) to[out=0,in=-100] (4,1);
    \draw[very thick] (2,1) to[out=-90,in=180] (2.5,.5) to[out=0,in=-90] (3,1);
    \draw[cutline] (2.5,-1) to (2.5,1.15);
    \end{tikzpicture}
    }\endxy
    \qquad\text{and}\qquad
    \xy (0,0)*{
    \begin{tikzpicture}[scale=0.65]
    \draw[thick,densely dotted] (-.5,1) to (4.5,1);
    \draw[black,fill=black] (0,1) circle (2.0pt);
    \draw[black,fill=black] (1,1) circle (2.0pt);
    \draw[black,fill=black] (2,1) circle (2.0pt);
    \draw[black,fill=black] (3,1) circle (2.0pt);
    \draw[black,fill=black] (4,1) circle (2.0pt);
    \draw[very thick] (0,1) to[out=-80,in=180] (1.5,-.3) to[out=0,in=-100] (3,1);
    \draw[very thick] (1,1) to[out=-90,in=180] (1.5,.5) to[out=0,in=-90] (2,1);
    \draw[very thick] (4,1) to (4,-1);
    \draw[cutline] (2.5,-1) to (2.5,1.15);
    \end{tikzpicture}
    }\endxy
    \end{equation*}
    Now, consider the two different flags
    \[
    (F_1,F_2,F_3,\mathrm{span}(e_1,e_2,e_3,f_1),\mathrm{span}(e_1,e_2,e_3,f_1,f_2))
    \]
    and 
    \[
    (F_1,F_2,F_3,\mathrm{span}(e_1,e_2,f_1,f_2),\mathrm{span}(e_1,e_2,e_3,f_1,f_2)).
    \]
    in $\spr_{(3,2)}$. By definition, the image of both of these flags under $\pi$ is $F_{\bullet}$. They belong to the irreducible components of $\spr_{(3,2)}$ corresponding to the respective diagrams
    
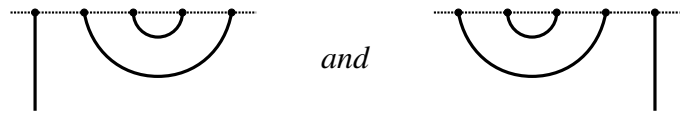
\begin{equation*}
    \xy (0,0)*{
    \begin{tikzpicture}[scale=0.65]
    \draw[thick,densely dotted] (-.5,1) to (4.5,1);
    \draw[black,fill=black] (0,1) circle (2.0pt);
    \draw[black,fill=black] (1,1) circle (2.0pt);
    \draw[black,fill=black] (2,1) circle (2.0pt);
    \draw[black,fill=black] (3,1) circle (2.0pt);
    \draw[black,fill=black] (4,1) circle (2.0pt);
    \draw[very thick] (0,1) to (0,-1);
    \draw[very thick] (1,1) to[out=-80,in=180] (2.5,-.3) to[out=0,in=-100] (4,1);
    \draw[very thick] (2,1) to[out=-90,in=180] (2.5,.5) to[out=0,in=-90] (3,1);
    \end{tikzpicture}
    }\endxy
    \qquad\text{and}\qquad
    \xy (0,0)*{
    \begin{tikzpicture}[scale=0.65]
    \draw[thick,densely dotted] (-.5,1) to (4.5,1);
    \draw[black,fill=black] (0,1) circle (2.0pt);
    \draw[black,fill=black] (1,1) circle (2.0pt);
    \draw[black,fill=black] (2,1) circle (2.0pt);
    \draw[black,fill=black] (3,1) circle (2.0pt);
    \draw[black,fill=black] (4,1) circle (2.0pt);
    \draw[very thick] (0,1) to[out=-80,in=180] (1.5,-.3) to[out=0,in=-100] (3,1);
    \draw[very thick] (1,1) to[out=-90,in=180] (1.5,.5) to[out=0,in=-90] (2,1);
    \draw[very thick] (4,1) to (4,-1);
    \end{tikzpicture}
    }\endxy
    \end{equation*}
    which are both in $\mathbb{B}_{3,2,2}$. Therefore, the restriction of $\pi$ to $\bigcup_{\ba\in \mathbb{B}_{3,2,2}}K_{\iota(\ba)}$ is not an isomorphism onto the $\Delta$-Springer variety $\dspr_{(3,2),2}$ since it is not bijective. Note that we cannot remove more irreducible components of the Springer fiber $\spr_{(3,2)}$ to make the map injective. Consider the three flags
    \begin{align*}
    (&\mathrm{span}(e_1),&&\mathrm{span}(e_1,e_2),&&\mathrm{span}(e_1,e_2,e_3)),\\
    (&\mathrm{span}(e_1+f_1),&&\mathrm{span}(e_1+f_1,e_2+f_2),&&\mathrm{span}(e_1,f_2,e_2+f_2)),\\
    (&\mathrm{span}(e_1+f_1),&&\mathrm{span}(e_1,f_1),&&\mathrm{span}(e_1,f_1,e_2-f_2)).
    \end{align*}
    in $\dspr_{(3,2),2}$. For each of these flags, one can check that they are in the image of a unique irreducible component of $\spr_{(3,2)}$. In particular, throwing out an additional irreducible component of $\spr_{(3,2)}$ would not yield a surjection onto $\dspr_{(3,2),2}$.
\end{exe}

\subsection{Comparison with the exotic Springer fiber for a one-row bipartition}
\label{sec:delta-exotic}

We now compare a two-row $\Delta$-Springer variety with an exotic Springer fiber associated with a one-row bipartition. Before doing so, let us recall the description of \cite{SW22} of the irreducible components of exotic Springer fibers for one-row bipartitions.

Firstly, \cite[Proposition 14]{SW22} shows that we can ignore the symplectic structure in the case of exotic Springer fibers for one-row bipartitions. If the bipartition is $((p),(q-p))$ then
\begin{gather}
    \label{eq:exotic_Y}
    \espr_{((p),(q-p))} \simeq \{ F_{\bullet} \in Y_q \mid e_p \in F_q\}.
\end{gather}
As for the $\Delta$-Springer variety, the notation $\espr_{((p),(q-p))}$ will refer to the exotic Springer fiber embedded in $Y_{q}$. We first deal with the extremal case $m=k$.

\begin{prop}\label{prop:exotic_as_Delta}
The $\Delta$-Springer variety $\dspr_{(n-k,k),k}$ and the exotic Springer fiber $\espr_{((n-2k),(k))}$ are isomorphic. 
\end{prop}

\begin{proof}
    Using an inductive argument as in \cite[Proposition 18]{SW22}, one can prove that if a flag  $F_{\bullet}\in\espr_{((p),(q-p))}$ then $F_{q} \subseteq \mathrm{span}(e_1,\ldots,e_{q},f_1,\ldots,f_{q-p})$. Therefore $\dspr_{(n-k,k),k}$ and $\espr_{((n-2k),(k))}$ are isomorphic to subvarieties of $\{F_{\bullet} \in Y_{n-k}\mid F_{n-k}\subseteq E_{n-k,k}\}$, see \eqref{eq:delta_Y} and \eqref{eq:exotic_Y}. By \autoref{lem:vectors_last}, we have
    \begin{gather*}
        \dspr_{(n-k,k),k}\cong\left\{ F_{\bullet} \in Y_{n-k} \mid F_{n-k}\subseteq E_{n-k,k}, e_{n-2k},f_{0}\in F_{n-k}\right\}.
    \end{gather*}
    Since $f_{0}=0$, comparing with $\eqref{eq:exotic_Y}$ shows that both varieties are equal as subvarieties of $Y_{n-k}$.
\end{proof}

We turn back to the general case of $0\leq m \leq k$. Elements of the $\Delta$-Springer variety $\dspr_{(n-k,k),m}$ are described by flags in $Y_{n-m}$ and so are the elements of the exotic Springer fiber $\espr_{((n-m-k),(k))}$.

We quickly recall the diagrammatics describing the irreducible components of the exotic Springer fiber associated with the bipartition $((n-m-k),(k))$, see \cite{SW22} for more details. These irreducible components are indexed by one-boundary diagrams on $n-m$ points which are endpoints of rays, cups or half-cups: cups connect two points, and both rays and half-cups connect only one point. Doing so, we require that the diagram is crossingless. We will denote by $\mathbb{B}_{((n-m-k),(k))}$ this set of diagrams with a total number of cups and half-cups equal to $k$. 

There is a map $\mathbb{B}_{n-k,k,m}$ into $\mathbb{B}_{((n-m-k),(k))}$ by deleting the part of the diagram right of the cut line. Since there is no cup among the last $m$ points of an element of $\mathbb{B}_{n-k,k,m}$, this map is an injection: we can reconstruct the initial diagram by completing the half-cups and then complete the remaining points with rays.

\begin{exe}
The previous injection is illustrated as below:
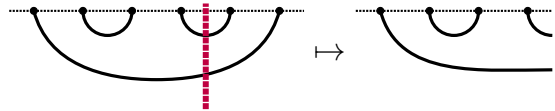
\begin{gather*}
    \xy (0,0)*{
    \begin{tikzpicture}[baseline={(0,0)},scale=0.65]
    \draw[thick,densely dotted] (-.5,1) to (5.5,1);
    \draw[black,fill=black] (0,1) circle (2.0pt);
    \draw[black,fill=black] (1,1) circle (2.0pt);
    \draw[black,fill=black] (2,1) circle (2.0pt);
    \draw[black,fill=black] (3,1) circle (2.0pt);
    \draw[black,fill=black] (4,1) circle (2.0pt);
    \draw[black,fill=black] (5,1) circle (2.0pt);
    \draw[very thick] (1,1) to[out=-90,in=180] (1.5,.5) to[out=0,in=-90] (2,1);
    \draw[very thick] (0,1) to[out=-75,in=180] (2.5,-.4) to[out=0,in=-105] (5,1);
    \draw[very thick] (3,1) to[out=-90,in=180] (3.5,.5) to[out=0,in=-90] (4,1);
    \draw[cutline] (3.5,-1) to (3.5,1.15);
    \end{tikzpicture}
    \mapsto
    \begin{tikzpicture}[baseline={(0,0)},scale=0.65]
    \draw[thick,densely dotted] (-.5,1) to (3.5,1);
    \draw[black,fill=black] (0,1) circle (2.0pt);
    \draw[black,fill=black] (1,1) circle (2.0pt);
    \draw[black,fill=black] (2,1) circle (2.0pt);
    \draw[black,fill=black] (3,1) circle (2.0pt);
    \draw[very thick] (1,1) to[out=-90,in=180] (1.5,.5) to[out=0,in=-90] (2,1);
    \draw[very thick] (0,1) to[out=-75,in=180] (3.5,-0.2);
    \draw[very thick] (3,1) to[out=-90,in=180] (3.5,.5);
    \end{tikzpicture}
    }\endxy
\end{gather*}
The following diagram is not in the image of the injection $\mathbb{B}_{3,3,2}\rightarrow \mathbb{B}_{((1),(3))}$:
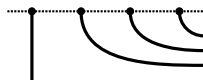
\begin{gather*}
    \begin{tikzpicture}[scale=0.65]
    \draw[thick,densely dotted] (-.5,1) to (3.5,1);
    \draw[black,fill=black] (0,1) circle (2.0pt);
    \draw[black,fill=black] (1,1) circle (2.0pt);
    \draw[black,fill=black] (2,1) circle (2.0pt);
    \draw[black,fill=black] (3,1) circle (2.0pt);
    \draw[very thick] (0,1) to (0,-0.4);
    \draw[very thick] (1,1) to[out=-90,in=180] (3.5,-.1);
    \draw[very thick] (2,1) to[out=-90,in=180] (3.5,.2);
    \draw[very thick] (3,1) to[out=-90,in=180] (3.5,.5);
    \end{tikzpicture}
\end{gather*}
Indeed, three points are needed to complete the half-cups and only two points are allowed on the right of the cut line of an element of 
$\mathbb{B}_{3,3,2}$.
\end{exe}

It easy to see that the image of this map is the subset of $\mathbb{B}_{((n-m-k),(k))}$ with at most $m$ half-cups. Using this identification, given $\ba\in\mathbb{B}_{n-k,k,m}$, we will denote by $K^{\mathrm{e}}_{\ba}$ the corresponding irreducible component of the exotic Springer fiber $\espr_{((n-m-k),(k))}$.

\begin{thm}\label{prop:comparison_exotic_delta}
    The (embedded) $\Delta$-Springer variety $\dspr_{(n-k,k),m}$ is equal to the closed subvariety $\bigcup_{\ba\in \mathbb{B}_{n-k,k,m}}K_{\ba}^{\mathrm{e}}$ of the (embedded) exotic Springer fiber $\espr_{((n-m-k),(k))}$. 
\end{thm}

\begin{proof}
    Since $e_{n-k-m}\in z^m(E_{n-k,k})$, the result follows from \cite[Proposition 14]{SW22} and the description of irreducible components of the exotic Springer fiber \cite[Theorem 15]{SW22}. 
\end{proof}

\section{A topological model and the action of the symmetric group}\label{sec:top-model-action-symmetric}

Recall that we have fixed a two-row partition $(n-k,k)$ and $0\leq m \leq k$. Using \autoref{prop:comparison_exotic_delta} and the topological model for the exotic Springer fiber of \cite{SW22}, we obtain a topological model of the $\Delta$-Springer variety $\dspr_{(n-k,k),m}$. We then deduce a skein theoretic description of the action of the symmetric group $S_{n-m}$ on the homology of the $\Delta$-Springer variety.

\subsection{A topological model}

Let $\mathbb{S}^2 \subseteq \mathbb{R}^3$ be the two dimensional standard unit sphere with north pole $p=(0,0,1)$. Given a $\Delta$-cup diagram $\ba\in\mathbb{B}_{n-k,k,m}$, define 
\[
S_\ba=\big\{(x_1,\ldots,x_{n-m})\in\left(\mathbb S^2\right)^{n-m}\mid x_j=-x_i\text{ if }i\CupConnect j,\text{ and }x_i=p\text{ if }\RayConnect\big\}.
\]
Note that $S_\ba$ is homeomorphic to a product of $2$-spheres. Each left endpoint of a cup in the diagram $\ba\in\mathbb{B}_{n-k,k,m}$ contributes exactly one sphere. 

\begin{defn}
The \emph{topological $\Delta$-Springer variety} $\mathcal{S}^{\Delta}_{(n-k,k),m}$ is defined as the union 
\[
\mathcal{S}^{\Delta}_{(n-k,k),m} := \bigcup_{\ba \in \mathbb{B}_{n-k,k,m}}S_\ba \subseteq\left(\mathbb S^2\right)^{n-m}.
\]
\end{defn}

The above definition of $S_{\ba}$ does not use the part of the diagram right of the cut line, and makes sense for any $\ba\in \mathbb{B}_{((n-m-k),(k))}$. 

\begin{prop}
    There exists a homeomorphism between the $\Delta$-Springer variety $\dspr_{(n-k,k),m}$ and the topological $\Delta$-Springer variety $\mathcal{S}^{\Delta}_{(n-k,k),m}$ such that the irreducible component $K_{\ba}$ of $\dspr_{(n-k,k),m}$ is sent to $S_{\ba}$. 
\end{prop}

\begin{proof}
In \cite{SW22}, a homeomorphism between the exotic Springer variety $\espr_{((n-m-k),(k))}$ and
\[
\mathcal{S}^{\mathrm{e}}_{((n-m-k),(k))} := \bigcup_{\ba \in \mathbb{B}_{((n-m-k),(k))}}S_\ba \subseteq\left(\mathbb S^2\right)^{n-m}
\]
is constructed and the irreducible component $K^{\mathrm{e}}_{\ba}$ corresponds to $S_{\ba}$. Therefore, using \autoref{prop:comparison_exotic_delta}, the restriction of the homeomorphism $\espr_{((n-m-k),(k))}\simeq\mathcal{S}^{\mathrm{e}}_{((n-m-k),(k))}$ to the closed subvarieties provides a homeomorphism $\dspr_{(n-k,k),m}\simeq\mathcal{S}^{\Delta}_{(n-k,k),m}$. The irreducible component $K_{\ba}$ of $\dspr_{(n-k,k),m}$ is then sent to $S_{\ba}$.
\end{proof}

Exactly as in \cite[Section 4.3]{SW22}, we describe the pairwise intersection of irreducible components in terms of circle diagrams.

For $\ba\in \mathbb{B}_{n-k,k,m}$, denote by $\overline{\ba}$ the reflection of $\ba$ along a horizontal line. For $\ba$ and $\mathbf{b}$ two $\Delta$-cup diagrams in $\mathbb{B}_{n-k,k,m}$, we call the concatenation $\overline{\ba}\mathbf{b}$ a \emph{circle diagram}. Such a diagram only contains circles and (open) lines. In a circle diagram, a line is called \emph{non-propagating} if it is entirely on the left of the cut line and its endpoints are on the same side of the horizontal line. Otherwise, a line is called \emph{propagating}. 

\begin{prop}
    Let $\ba,\mathbf{b}\in\mathbb{B}_{n-k,k,m}$. The intersection $S_{\ba}\cap S_{\mathbf{b}}$ is nonempty if and only if all the lines in the circle diagram $\overline{\ba}\mathbf{b}$ are propagating. Moreover, $S_{\ba}\cap S_{\mathbf{b}}$ is homeomorphic to $(\mathbb{S}^{2})^{\ell}$, where $\ell$ counts the number of circles plus the number of open lines with both endpoints on the cut line, in the diagram obtained from $\overline{\ba\mathbf{b}}$ by erasing everything on the right of the cut line.
\end{prop}

\begin{proof}
    This follows from \cite[Theorem 29]{SW22}.
\end{proof}

\begin{exe}
Below we give examples of cup diagrams $\ba$, $\mathbf{b}$ and a circle diagram $\overline{\ba}\mathbf{b}$. 
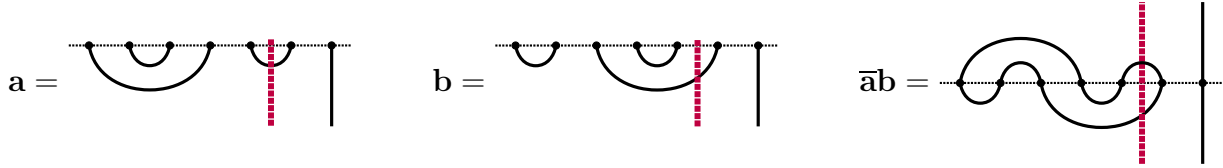
\[
\ba =
\xy (0,0)*{
\begin{tikzpicture}[scale=0.535,yscale=-1]
\draw[thick,densely dotted] (-1.5,1) to (5.5,1);
\draw[black,fill=black] (-1,1) circle (2.5pt); 
\draw[black,fill=black] (0,1) circle (2.5pt);
\draw[black,fill=black] (1,1) circle (2.5pt);
\draw[black,fill=black] (2,1) circle (2.5pt);
\draw[black,fill=black] (3,1) circle (2.5pt);
\draw[black,fill=black] (4,1) circle (2.5pt);
\draw[black,fill=black] (5,1) circle (2.5pt);
\draw[very thick] (-1,1) to[out=80,in=180] (0.5,2.1) to[out=0,in=100] (2,1);
\draw[very thick] (0,1) to[out=80,in=180] (0.5,1.5) to[out=0,in=100] (1,1);
\draw[very thick] (3,1) to[out=80,in=180] (3.5,1.5) to[out=0,in=100] (4,1);
\draw[very thick] (5,1) to (5,3);
\draw[cutline] (3.5,.85) to (3.5,3);
\end{tikzpicture}
}\endxy
\mspace{45mu}
\mathbf{b} = \xy (0,0)*{
\begin{tikzpicture}[scale=0.535]
\draw[thick,densely dotted] (-1.5,1) to (5.5,1);
\draw[black,fill=black] (-1,1) circle (2.5pt);
\draw[black,fill=black] (0,1) circle (2.5pt);
\draw[black,fill=black] (1,1) circle (2.5pt);
\draw[black,fill=black] (2,1) circle (2.5pt);
\draw[black,fill=black] (3,1) circle (2.5pt);
\draw[black,fill=black] (4,1) circle (2.5pt);
\draw[black,fill=black] (5,1) circle (2.5pt);
\draw[very thick] (-1,1) to[out=-80,in=180] (-.5,.5) to[out=0,in=-100] (0,1);
\draw[very thick] (1,1) to[out=-80,in=180] (2.5,-.1) to[out=0,in=-100] (4,1);
\draw[very thick] (2,1) to[out=-80,in=180] (2.5,.5) to[out=0,in=-100] (3,1);
\draw[very thick] (5,-1) to (5,1);
\draw[cutline] (3.5,-1) to (3.5,1.15);
\end{tikzpicture}
}\endxy
\mspace{45mu}
\overline{\ba}\mathbf{b} = \xy (0,0)*{
\begin{tikzpicture}[scale=0.535]
\begin{scope}[shift={(0,0)}]
\draw[thick,densely dotted] (-1.5,1) to (5.5,1);
\draw[black,fill=black] (-1,1) circle (2.5pt); 
\draw[black,fill=black] (0,1) circle (2.5pt);
\draw[black,fill=black] (1,1) circle (2.5pt);
\draw[black,fill=black] (2,1) circle (2.5pt);
\draw[black,fill=black] (3,1) circle (2.5pt);
\draw[black,fill=black] (4,1) circle (2.5pt);
\draw[black,fill=black] (5,1) circle (2.5pt);
\draw[very thick] (-1,1) to[out=-80,in=180] (-.5,.5) to[out=0,in=-100] (0,1);
\draw[very thick] (1,1) to[out=-80,in=180] (2.5,-.1) to[out=0,in=-100] (4,1);
\draw[very thick] (2,1) to[out=-80,in=180] (2.5,.5) to[out=0,in=-100] (3,1);
\draw[very thick] (5,-1) to (5,1);
\draw[cutline] (3.5,-1) to (3.5,3);
\end{scope}
\begin{scope}[shift={(0,0)}]
\draw[very thick] (-1,1) to[out=80,in=180] (0.5,2.1) to[out=0,in=100] (2,1);
\draw[very thick] (0,1) to[out=80,in=180] (0.5,1.5) to[out=0,in=100] (1,1);
\draw[very thick] (3,1) to[out=80,in=180] (3.5,1.5) to[out=0,in=100] (4,1);
\draw[very thick] (5,1) to (5,3);
\end{scope}
\end{tikzpicture}
}\endxy
\]
Therefore the intersection $S_{\ba}\cap S_{\mathbf{b}}$ is a $2$-sphere.
\end{exe}

\subsection{Diagrammatic description of the homology}

The topological model constructed in the previous section naturally embeds in $(\mathbb{S}^2)^{n-m}$ and this inclusion induces a map 
\begin{gather}
H_{*}(\mathcal{S}^{\Delta}_{(n-k,k),m}) \hookrightarrow H_{*}((\mathbb{S}^2)^{n-m}).
\end{gather}

We describe explicitly this map using line diagrams and show that this map is injective and the natural action of the symmetric group $S_{n-m}$ on $H_{*}((\mathbb{S}^2)^{n-m})$ stabilizes the subspace $H_{*}(\mathcal{S}^{\Delta}_{\lambda,m})$.

The decomposition of the sphere $\mathbb{S}^2$ as $\{p\}\cup(\mathbb{S}^2\setminus\{p\})$ defines a CW-structure on $\mathbb{S}^2$ with one $0$-cell and one $2$-cell. We fix this CW-structure and then equip $(\mathbb{S}^2)^{n-m}$ with the Cartesian product CW-structure. The cells of this CW-structure are indexed by the subsets $U$ of $\{1,\ldots,n-m\}$ by mapping $U$ to the cell $C_U$ which is defined by choosing the $0$-cell in the $i$-th component if $i\not\in U$ and the $2$-cell if $i\in U$. We will denote by $l_U$ the homology class of the cell $C_U$ in $H_{*}((\mathbb{S}^2)^{n-m})$ so that $(l_U)_{U}$ is a basis of $H_{*}((\mathbb{S}^2)^{n-m})$, the degree of $l_U$ is given by twice the cardinality of $U$.

Following \cite[Section 6.1]{SW22}, we will write the elements $l_U$ as line diagrams: we attach $n-m$ vertical lines, which are decorated by (empty) dots on the lines with endpoints not in $U$.

\begin{exe}
    If $U=\{1,3\}\subseteq \{1,2,3,4\}$, the corresponding line diagram is
    \begin{gather*}
        l_U =
        \begin{tikzpicture}[baseline={(0,-.5)},scale=.8]
        \draw[thick] (0,0) -- +(0,-1);
        \draw[thick] (.5,0) -- +(0,-1);
        \draw[thick] (1,0) -- +(0,-1);
        \draw[thick] (1.5,0) -- +(0,-1);
        \draw (0.5,-.5) circle(3pt);
        \draw (1.5,-.5) circle(3pt);
        \end{tikzpicture}\ .
    \end{gather*}
\end{exe}

Given a $\Delta$-weight $\alpha$ of type $(n-k,k,m)$, we define an element $L_{\alpha}\in H_{*}((\mathbb{S}^2)^{n-m})$ as follows. Denote by $\mathcal{U}_{\alpha}$ the set of all $U\subseteq \{1,\ldots,n-m\}$ containing one endpoint of each cup of $C(\alpha)$ left of the cut line and the left endpoints of cups which cross the cut line. We set
\begin{gather*}
    L_{\alpha}=\sum_{U\in \mathcal{U}_{\alpha}}(-1)^{\Lambda_{\alpha}(U)}l_U,
\end{gather*}
where $\Lambda_{\alpha}(U)$ is the number of right endpoints of cups of $C(\alpha)$ in $U$. Note that $L_{\alpha}$ is in degree $2d$, where $d$ is the number of cups in $C(\alpha)$.

\begin{exe}\label{ex:line-diagram}
    Consider the $\Delta$-weight $\alpha=\down\up\up\down\vert\!\up\down$. The associated diagram is
    \begin{gather*}
            C(\alpha) = 
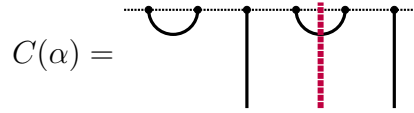
\xy (0,0)*{
            \begin{tikzpicture}[scale=0.65]
            \draw[thick,densely dotted] (-.5,1) to (5.5,1);
            \draw[black,fill=black] (0,1) circle (2.0pt);
            \draw[black,fill=black] (1,1) circle (2.0pt);
            \draw[black,fill=black] (2,1) circle (2.0pt);
            \draw[black,fill=black] (3,1) circle (2.0pt);
            \draw[black,fill=black] (4,1) circle (2.0pt);
            \draw[black,fill=black] (5,1) circle (2.0pt);
            \draw[very thick] (0,1) to[out=-90,in=180] (0.5,.5) to[out=0,in=-90] (1,1);
            \draw[very thick] (3,1) to[out=-90,in=180] (3.5,.5) to[out=0,in=-90] (4,1);
            \draw[very thick] (2,1) -- (2,-1);
            \draw[very thick] (5,1) -- (5,-1);
            \draw[cutline] (3.5,-1) to (3.5,1.15);
            \end{tikzpicture}
            }\endxy 
        \end{gather*}
        We then have
        \begin{gather*}
        L_{\alpha} = 
        \begin{tikzpicture}[baseline={(0,-.5)},scale=.8]
        \draw[thick] (0,0) -- +(0,-1);
        \draw[thick] (.5,0) -- +(0,-1);
        \draw[thick] (1,0) -- +(0,-1);
        \draw[thick] (1.5,0) -- +(0,-1);
        \draw (.5,-.5) circle(3pt);
        \draw (1,-.5) circle(3pt);
        \end{tikzpicture}
        -
        \begin{tikzpicture}[baseline={(0,-.5)},scale=.8]
        \draw[thick] (0,0) -- +(0,-1);
        \draw[thick] (.5,0) -- +(0,-1);
        \draw[thick] (1,0) -- +(0,-1);
        \draw[thick] (1.5,0) -- +(0,-1);
        \draw (0,-.5) circle(3pt);
        \draw (1,-.5) circle(3pt);
        \end{tikzpicture}\, ,
        \end{gather*}
        the first term corresponding to $U=\{1,4\}$ and the second one to $U=\{2,4\}$.
\end{exe}

\begin{prop}
The inclusion $\mathcal{S}^{\Delta}_{(n-k,k),m}\hookrightarrow (\mathbb{S}^{2})^{n-m}$ induces an embedding 
\[H_{*}(\mathcal{S}^{\Delta}_{(n-k,k),m}) \hookrightarrow H_{*}((\mathbb{S}^2)^{n-m})\] 
which sends the homology class of $K_{\alpha}$ to $L_{\alpha}$.
\end{prop}

\begin{proof}
  The proof is similar to \cite[Proposition 57]{SW22}. In order to sketch the idea, consider the commutative diagram
  \[
    \begin{xy}
      \xymatrix{
        \bigoplus\limits_{{\bf a}\in\mathbb{B}_{n-k,k,m}} H_*(S_{\bf a}) \ar[r] \ar@/^1.5pc/[rr] & H_*(\mathcal{S}^{\Delta}_{(n-k,k),m}) \ar[r] & H_*((\mathbb S^2)^{n-m}).
      }
    \end{xy}
  \]
  All maps are induced by the natural inclusions. Since $K_{\alpha}\subset S_{\bf a}$ for some ${\bf a}\in\mathbb{B}_{n-k,k,m}$, one can apply the upper map to the homology class of $K_{\alpha}$ and check that its image in $H_*((\mathbb S^2)^{n-m})$ is the line diagram sum $L_{\alpha}$. The left horizontal map sends the homology class of $K_{\alpha}$ in $H_*(S_{\bf a})$ to the homology class of $K_{\alpha}$ in $H_*(\mathcal{S}^{\Delta}_{(n-k,k),m})$. Hence, by the commutativity of the diagram, the homology class of $K_{\alpha}$ in $H_*(\mathcal{S}^{\Delta}_{(n-k,k),m})$ gets sent to $L_{\alpha}$. Since the homology classes of the $K_{\alpha}$ form a basis of $H_*(\mathcal{S}^{\Delta}_{(n-k,k),m})$, and the line diagram sums $L_{\alpha}$, where $\alpha$ varies over all $\Delta$-weights of type $(n-k,k,m)$, can be shown to be linearly independent, it follows that the map $H_{*}(\mathcal{S}^{\Delta}_{(n-k,k),m}) \hookrightarrow H_{*}((\mathbb{S}^2)^{n-m})$ is injective.
\end{proof}

\subsection{Action of the symmetric group}

The symmetric group $S_{n-m}$ acts naturally on $(\mathbb{S}^2)^{n-m}$ by permutation. This induces an action at the level of the homology, which is again described by permutations of the lines in a line diagram. From now on, we identify the homology $H_{*}(\mathcal{S}^{\Delta}_{(n-k,k),m})$ with its image in $H_{*}((\mathbb{S}^{2})^{n-m})$.

We give a skein theoretic description of the action of the symmetric group, thereby answering \cite[Question 8.6]{GLW21}. We write $s_i\in S_{n-m}$ for the transposition $(i, i+1)$. To state the explicit action, we will identify the basis element $L_{\alpha}$ of the homology corresponding to the $\Delta$-weight $\alpha$ with the corresponding $\Delta$-cup diagram $C(\alpha)$.

\begin{prop}\label{prop:action_symmetric}
    The subspace $H_{*}(\mathcal{S}^{\Delta}_{(n-k,k),m})$ is stable under the action of the symmetric group. Moreover, the action of $s_i$ on $L_{\alpha}$ is obtained by stacking the diagram
    \begin{gather*}
        \begin{tikzpicture}
            \draw[thick] (0,0) -- +(0,-1);
            \draw[thick] (.75,0) -- +(0,-1);
            \draw[thick] (1.25,0) -- +(.5,-1);
            \draw[thick] (1.75,0) -- +(-.5,-1);
            \draw[thick] (2.25,0) -- +(0,-1);
            \draw[thick] (3,0) -- +(0,-1);
            \node at (.4,-.5) {$\cdots$};
            \node at (2.65,-.5) {$\cdots$};
            \node[above] at (1.20,0.03) {\tiny{$i$}};
            \node[above] at (1.75,0) {\tiny{$i+1$}};
        \end{tikzpicture}
    \end{gather*}
    on top of $C(\alpha)$ and using the following relations
    \begin{itemize}
        \item isotopies,
        \item skein relation: $\begin{tikzpicture}[baseline={(0,-.5)}, scale=.7]
            \draw[thick] (0,0) -- +(.5,-1);
            \draw[thick] (.5,0) -- +(-.5,-1);
        \end{tikzpicture}
        =
        \begin{tikzpicture}[baseline={(0,-.5)}, scale=.7]
            \draw[thick] (0,0) to[out=-90,in=180] (.25,-.25) to[out=0,in=-90] (.5,0);
            \draw[thick] (0,-1) to[out=90,in=180] (.25,-.75) to[out=0,in=90] (.5,-1);
        \end{tikzpicture}
        +
        \begin{tikzpicture}[baseline={(0,-.5)}, scale=.7]
            \draw[thick] (0,0) -- +(0,-1);
            \draw[thick] (.5,0) -- +(0,-1);
        \end{tikzpicture},
        $
        \item bubble removal:
        $
        \begin{tikzpicture}[baseline={(0,-.1)}, scale=.7]
            \draw[thick] (0,0) circle (.4);
        \end{tikzpicture}
        = -2,
        $
        \item cap removal: 
        $
        \begin{tikzpicture}[baseline={(0,-.5)}, scale=.7]
            \draw[thick] (0,-1) to (0,-.5) to[out=90,in=180] (.25,-.25) to[out=0,in=90] (.5,-.5) to (.5,-1);
        \end{tikzpicture}
        =
        0,$
        \item if these operations create diagrams which are not allowed (\emph{i.e.} with an entire cup right of the cut line), then these diagrams are set to zero.
    \end{itemize}
\end{prop}

\begin{exe}
    Consider the weight $\alpha$ of \autoref{ex:line-diagram}, with corresponding $\Delta$-cup diagram
    \begin{gather*}
            C(\alpha) = \xy (0,0)*{
            \begin{tikzpicture}[scale=0.65]
            \draw[thick,densely dotted] (-.5,1) to (5.5,1);
            \draw[black,fill=black] (0,1) circle (2.0pt);
            \draw[black,fill=black] (1,1) circle (2.0pt);
            \draw[black,fill=black] (2,1) circle (2.0pt);
            \draw[black,fill=black] (3,1) circle (2.0pt);
            \draw[black,fill=black] (4,1) circle (2.0pt);
            \draw[black,fill=black] (5,1) circle (2.0pt);
            \draw[very thick] (0,1) to[out=-90,in=180] (0.5,.5) to[out=0,in=-90] (1,1);
            \draw[very thick] (3,1) to[out=-90,in=180] (3.5,.5) to[out=0,in=-90] (4,1);
            \draw[very thick] (2,1) -- (2,-1);
            \draw[very thick] (5,1) -- (5,-1);
            \draw[cutline] (3.5,-1) to (3.5,1.15);
        \end{tikzpicture}
        }\endxy
    \end{gather*}
    Then the actions of $s_1$, $s_2$ and $s_3$ on $C(\alpha)$ are respectively given by
    \begin{gather*}
        -\begin{tikzpicture}[baseline={(0,0)},scale=0.65]
            \draw[thick,densely dotted] (-.5,1) to (5.5,1);
            \draw[black,fill=black] (0,1) circle (2.0pt);
            \draw[black,fill=black] (1,1) circle (2.0pt);
            \draw[black,fill=black] (2,1) circle (2.0pt);
            \draw[black,fill=black] (3,1) circle (2.0pt);
            \draw[black,fill=black] (4,1) circle (2.0pt);
            \draw[black,fill=black] (5,1) circle (2.0pt);
            \draw[very thick] (0,1) to[out=-90,in=180] (0.5,.5) to[out=0,in=-90] (1,1);
            \draw[very thick] (3,1) to[out=-90,in=180] (3.5,.5) to[out=0,in=-90] (4,1);
            \draw[very thick] (2,1) -- (2,-1);
            \draw[very thick] (5,1) -- (5,-1);
            \draw[cutline] (3.5,-1) to (3.5,1.15);
        \end{tikzpicture}\ ,\quad
        \begin{tikzpicture}[baseline={(0,0)},scale=0.65]
            \draw[thick,densely dotted] (-.5,1) to (5.5,1);
            \draw[black,fill=black] (0,1) circle (2.0pt);
            \draw[black,fill=black] (1,1) circle (2.0pt);
            \draw[black,fill=black] (2,1) circle (2.0pt);
            \draw[black,fill=black] (3,1) circle (2.0pt);
            \draw[black,fill=black] (4,1) circle (2.0pt);
            \draw[black,fill=black] (5,1) circle (2.0pt);
            \draw[very thick] (1,1) to[out=-90,in=180] (1.5,.5) to[out=0,in=-90] (2,1);
            \draw[very thick] (3,1) to[out=-90,in=180] (3.5,.5) to[out=0,in=-90] (4,1);
            \draw[very thick] (0,1) -- (0,-1);
            \draw[very thick] (5,1) -- (5,-1);
            \draw[cutline] (3.5,-1) to (3.5,1.15);
        \end{tikzpicture}
        +
        \begin{tikzpicture}[baseline={(0,0)},scale=0.65]
            \draw[thick,densely dotted] (-.5,1) to (5.5,1);
            \draw[black,fill=black] (0,1) circle (2.0pt);
            \draw[black,fill=black] (1,1) circle (2.0pt);
            \draw[black,fill=black] (2,1) circle (2.0pt);
            \draw[black,fill=black] (3,1) circle (2.0pt);
            \draw[black,fill=black] (4,1) circle (2.0pt);
            \draw[black,fill=black] (5,1) circle (2.0pt);
            \draw[very thick] (0,1) to[out=-90,in=180] (0.5,.5) to[out=0,in=-90] (1,1);
            \draw[very thick] (3,1) to[out=-90,in=180] (3.5,.5) to[out=0,in=-90] (4,1);
            \draw[very thick] (2,1) -- (2,-1);
            \draw[very thick] (5,1) -- (5,-1);
            \draw[cutline] (3.5,-1) to (3.5,1.15);
        \end{tikzpicture}\ ,\quad 
\end{gather*}
and
\begin{gather*}
\begin{tikzpicture}[baseline={(0,0)},scale=0.65]
            \draw[thick,densely dotted] (-.5,1) to (5.5,1);
            \draw[black,fill=black] (0,1) circle (2.0pt);
            \draw[black,fill=black] (1,1) circle (2.0pt);
            \draw[black,fill=black] (2,1) circle (2.0pt);
            \draw[black,fill=black] (3,1) circle (2.0pt);
            \draw[black,fill=black] (4,1) circle (2.0pt);
            \draw[black,fill=black] (5,1) circle (2.0pt);
            \draw[very thick] (0,1) to[out=-90,in=180] (0.5,.5) to[out=0,in=-90] (1,1);
            \draw[very thick] (2,1) to[out=-90,in=180] (2.5,.5) to[out=0,in=-90] (3,1);
            \draw[very thick] (4,1) -- (4,-1);
            \draw[very thick] (5,1) -- (5,-1);
            \draw[cutline] (3.5,-1) to (3.5,1.15);
        \end{tikzpicture}
        +
       \begin{tikzpicture}[baseline={(0,0)},scale=0.65]
            \draw[thick,densely dotted] (-.5,1) to (5.5,1);
            \draw[black,fill=black] (0,1) circle (2.0pt);
            \draw[black,fill=black] (1,1) circle (2.0pt);
            \draw[black,fill=black] (2,1) circle (2.0pt);
            \draw[black,fill=black] (3,1) circle (2.0pt);
            \draw[black,fill=black] (4,1) circle (2.0pt);
            \draw[black,fill=black] (5,1) circle (2.0pt);
            \draw[very thick] (0,1) to[out=-90,in=180] (0.5,.5) to[out=0,in=-90] (1,1);
            \draw[very thick] (3,1) to[out=-90,in=180] (3.5,.5) to[out=0,in=-90] (4,1);
            \draw[very thick] (2,1) -- (2,-1);
            \draw[very thick] (5,1) -- (5,-1);
            \draw[cutline] (3.5,-1) to (3.5,1.15);
        \end{tikzpicture}.
\end{gather*}
\end{exe}

\begin{proof}[Proof of \autoref{prop:action_symmetric}]
    This is a direct computation, along the lines of \cite{russell}.
\end{proof}

As already shown in \cite{GLW21} for the top degree, the representations obtained in each degree are not simple. In the following, if $\mu$ is a skew partition of $n-m$, we denote by $V_{\mu}$ the corresponding skew Specht module which is a complex representation of $S_{n-m}$. If $\mu$ is a partition of $n-m$, then it is the usual irreducible Specht module.

\begin{thm}\label{prop:action_symmetric_hom}
    Let $0\leq d \leq k$. As $S_{n-m}$ representations, we have the following isomorphisms:
    \begin{gather*}
        H_{2d(\mathcal{S}^{\Delta}_{(n-k,k),m})} \simeq \bigoplus_{j=\max(0,2d+m-n)}^{\min(m,d)} V_{(n-m-d+j,d-j)} \simeq V_{(n-d,d)/(m)}.
    \end{gather*}
\end{thm}

\begin{proof}
    For $0\leq j \leq \min(d,m)$, we consider the subspace $W_j$ of $H_{2d(\mathcal{S}^{\Delta}_{(n-k,k),m})}$ spanned by the vectors $L_{\alpha}$ such that $C(\alpha)$ has at most $j$ cups through the cut line. 

    Let us first notice that $W_{\min(d,m)}=H_{2d(\mathcal{S}^{\Delta}_{(n-k,k),m})}$. Indeed, if $L_{\alpha}\in H_{2d(\mathcal{S}^{\Delta}_{(n-k,k),m})}$ then the diagram $C(\alpha)$ has $d$ cups and since there are $m$ points on the right of the cut line, there are at most $m$ cups of $C(\alpha)$ through the cut line. 

    We also have $W_{j}=\{0\}$ if $j<m-n+2d$. Suppose that $m-n+2d>0$, otherwise there is nothing to prove. If $C$ is a cup diagram on $n$ points with $d$ cups, then there are $n-2d$ rays on this diagram. Denote by $r$ the number of rays on the right of the cut line. This implies that $m-r$ cups must pass through the cut line. But since $r\leq n-2d$, we obtain that $C$ must have at least $2d+m-n$ cups through the cut line.

    From the description of the action of $S_{n-m}$ in \autoref{prop:action_symmetric}, it is clear that $W_{j}$ is stable under the action of $S_{n-m}$. Therefore, we obtain a filtration
    \begin{gather*}
        \{0\} \subset W_{\max(0,2d+m-n)}\subset W_{\max(0,2d+m-n)+1} \subset \cdots \subset W_{\min(d,m)} = H_{2d(\mathcal{S}^{\Delta}_{(n-k,k),m})}
    \end{gather*}
    by $S_{n-m}$ invariant subspaces. We claim that $W_{j}/W_{j-1}\simeq V_{(n-m-d+j,d-j)}$, which will prove the first isomorphism of the proposition. 
    
    Indeed, the quotient space $W_j/W_{j-1}$ has a basis given by $\Delta$-cup diagrams on $n$ points with $d$ cups and exactly $j$ cups through the cut line. Forgetting the right part of such a diagram and replacing the $j$ half-cups created this way by rays, we obtain a cup diagram on $n-m$ points with $d-j$ cups, and all such diagrams can be obtained by this process.

    We obtain an isomorphism of vector spaces between $W_{j}/W_{j-1}$ and $H_{2(d-j)(\mathcal{S}_{(n-m-d+j,d-j)})}$, which is easily checked to be $S_{n-m}$-equivariant thanks to \autoref{prop:action_symmetric}. Here, $\mathcal{S}_{(n-m-d+j,d-j)}$ is the topological model for the Springer fiber for the partition $(n-m-d+j,d-j)$, see \cite{russell}. Therefore, since $H_{2(d-j)(\mathcal{S}_{(n-m-d+j,d-j)})}$ is isomorphic to $V_{(n-m-d+j,d-j)}$ as a representation of $S_{n-m}$, we have proven our claim.

    Concerning the isomorphism with the skew Specht module $V_{(n-d,d)/(m)}$, we use that the multiplicity of $V_{\mu}$ in $V_{(n-d,d)/(m)}$ as an $S_{n-m}$-representation is equal to the multiplicity of $V_{(n-d,d)}$ in $\mathrm{Ind}_{S_{n-m}\times S_m}^{S_n}(V_\mu\otimes V_{(m)})$ as $S_n$-representations \cite[Section 3]{jp79}. From Pieri's formula, we deduce that, as an $S_{n-m}$-representation, we have $V_{(n-d,d)/(m)}\simeq\bigoplus_{\mu}V_{\mu}$, the direct sum being on partitions $\mu$ obtained from $(n-d,d)$ by removing $m$ boxes, no two in the same column. We easily check that we obtain the expected direct sum.
\end{proof}

Note that the isomorphism $H_{*}(\dspr_{(n-k,k),m})\simeq H_{*}(\mathcal{S}^{\Delta}_{(n-k,k),m})$ intertwines the $S_{n-m}$-action since the first Chern class of the dual $i$th line bundle maps to the
hyperplane class of the $i$th copy of $\mathbb{P}^1 \simeq \mathbb{S}^2$, see \cite[Theorem 2.1]{CK08}.

\section{Action of the degenerate affine Hecke algebra}\label{sec:degenerate-affine}

We enhance the action of the symmetric group $S_{n-m}$ on the cohomology of the $\Delta$-Springer variety to an action of the degenerate affine Hecke algebra. Each degree of the cohomology will be irreducible for the action of this algebra.

\subsection{Degenerate affine Hecke algebra and action on the homology of \texorpdfstring{$(\mathbb{S}^{2})^{n-m}$}{(S2)(n-m)}}\label{S:Haction}

We start by defining the degenerate affine Hecke algebra.

\begin{defn}
    The degenerate affine Hecke algebra $H_{n-m}$ is the $\mathbb{C}$-algebra with generators  
    $\sigma_1,\ldots,\sigma_{n-m-1}$, $x_1,\ldots,x_{n-m}$ and relations
    \begin{align*}
        \sigma_i^2&=1 &\text{for } 1\leq i < n-m\\ 
        \sigma_i\sigma_j&=\sigma_j\sigma_i &\text{if } \lvert i-j \rvert > 1,\\
        \sigma_i\sigma_{i+1}\sigma_i &= \sigma_{i+1}\sigma_i\sigma_{i+1} & \text{for } 1 \leq i < n-m-1,\\
        x_ix_j&=x_jx_i & \text{for } 1 \leq i,j \leq n-m,\\
        \sigma_ix_j &= x_j\sigma_i &\text{if } j\neq i,i+1,\\
        x_{i+1}\sigma_i&-\sigma_ix_i = 1 & \text{for } 1 \leq i < n-m.
    \end{align*}
\end{defn}

There is a well-known polynomial action of $H_{n-m}$ on $\mathbb{C}[X_1,\ldots,X_{n-m}]$ given by the so-called Dunkl operators. This action preserves the degree, and the following definition is the restriction of the action of Dunkl operators on $\mathrm{span}(X_{i_1}\cdots X_{i_k} \mid k\in \mathbb{N}, 1 \leq i_1 < \cdots < i_k \leq n-m)$, written in the language of line diagrams.

\begin{defn}
    Let $\boldsymbol{\xi} = (\xi_0,\ldots,\xi_{n-m})\in \mathbb{C}^{n-m+1}$. Given a line diagram $l_U$, with $U$ a subset of $\{1,\ldots,n-m\}$ of cardinality $d$, and $1\leq i \leq n-m$, we define
    \begin{gather*}
    \mathcal{D}_i^{(\boldsymbol{\xi})}(l_U) 
    = 
    \begin{cases}\displaystyle
    (\xi_d+i-n+m-\left\lvert\{1 \leq j < i\mid j\not\in U\}\right\rvert)l_U  -\sum_{\substack{j=i+1\\j\not\in U}}^{n-m} l_{U\setminus\{i\}\cup\{j\}}& \text{if } i\in U,\\ \displaystyle 
    \left(i-n+m+\left\lvert\{i < j \leq n-m\mid j\in U\}\right\rvert\right) l_U  + \sum_{\substack{j=1\\j\in U}}^{i-1} l_{U\setminus\{j\}\cup\{i\}} & \text{if } i\not\in U.
    \end{cases}
    \end{gather*}
\end{defn}

We will call $\boldsymbol{\xi}$ the parameters of the action. Since this action arises from Dunkl operators, this defines an action of the degenerate affine Hecke algebra on $H_{*}((\mathbb{S}^2)^{n-m})$, see \cite{che05}.

\begin{prop}
    The assignment $\sigma_i\mapsto s_{i}$ and $x_i\mapsto \mathcal{D}_{i}^{(\boldsymbol{\xi})}$ is a well-defined action of the degenerate affine Hecke algebra $H_{n-m}$ on $H_{*}((\mathbb{S}^2)^{n-m})$.
\end{prop}

\subsection{Restriction to the homology of the \texorpdfstring{$\Delta$}{Delta}-Springer variety}

It turns out that the cohomology of the $\Delta$-Springer variety, viewed as a subspace of the cohomology of the product of $n-m$ spheres, is stable under the action of the degenerate affine Hecke algebra for specific values of the parameter $\boldsymbol{\xi}$.

\begin{thm}\label{thm:stability_degenerate}
    Let $\boldsymbol{\xi}=(\xi_0,\ldots,\xi_{n-m})$ with $\xi_d = n+1-d$. Then the subspace $H_{*}(\dspr_{(n-k,k),m})$ of $H_{*}((\mathbb{S}^2)^{n-m})$ is stable under the action of $H_{n-m}$.
\end{thm}

We will prove the theorem as follows: since the degenerate affine Hecke algebra $H_{n-m}$ is generated by the symmetric group and $x_{n-m}$, it suffices to show that the subspace $H_{*}(\dspr_{(n-k,k),m})$ is stable under the action of the symmetric group and under the action of $x_{n-m}$. The action of the symmetric group is given in \autoref{prop:action_symmetric}, therefore it suffices to consider the action of $x_{n-m}$. We give explicit formulas in terms of line diagrams, which have a skein theoretic interpretation akin to the action of the symmetric group.

Let $\alpha$ be a $\Delta$-weight of type $(n-k,k,m)$ and suppose that the corresponding $\Delta$-cup diagram $C(\alpha)$ has $d$ cups. In other words, the element $L_{\alpha}$ corresponding to $\alpha$ is in $H_{2d}(\dspr_{(n-k,k),m})$. We discuss three cases, depending whether $n-m$ is the endpoint of a ray, the right endpoint of a cup or the left endpoint of a cup in $C(\alpha)$. 

\begin{lem}\label{lem:action_dunkl_ray}
    Suppose that $n-m$ is the endpoint of a ray in $C(\alpha)$. Then $x_{n-m}\cdot L_{\alpha} = 0$.
\end{lem}

\begin{proof}
    By definition, we have
    \begin{gather*}
        L_{\alpha} = \sum_{U\in \mathcal{U_{\alpha}}} (-1)^{\Lambda_{\alpha}(U)}l_{U}.
    \end{gather*}

    Since we have supposed that $n-m$ is the endpoint of a ray in $C(\alpha)$, we know that $n-m\not\in U$ for all $U\in \mathcal{U}_{\alpha}$. Therefore, the definition of the action of the degenerate affine Hecke algebra via Dunkl operators implies that
    \begin{gather*}
        x_{n-m}\cdot L_{\alpha} 
        = \sum_{U\in\mathcal{U}_{\alpha}} (-1)^{\Lambda_{\alpha}(U)}\sum_{i\in U} l_{U\setminus\{i\}\cup\{n-m\}}
        =\sum_{\substack{U\in\mathcal{U}_{\alpha}\\i\in U}}(-1)^{\Lambda_{\alpha}(U)}l_{U\setminus\{i\}\cup\{n-m\}}.
    \end{gather*}

    We consider the bijection $\psi$ of $\{(U,i)\mid U\in \mathcal{U}_{\alpha}, i\in U\}$ given by $\psi(U,i) = (U\setminus\{i\}\cup\{j\},j)$ where $j$ is such that $i$ and $j$ are connected by a cup in $C(\alpha)$. This is well defined: there exists no cup that crosses the cut line because $n-m$ is the endpoint of a ray. Moreover, it is clear that $\Lambda_{\alpha}(U\setminus\{i\}\cup\{j\}) = \Lambda_{\alpha}(U) \pm 1$ if $i\in U$ and $j$ is connected to $i$ by a cup in $C(\alpha)$.

    Therefore, using the bijection $\psi$, we obtain that
    \begin{gather*}
        x_{n-m}\cdot L_{\alpha} 
        = \sum_{\substack{U\in\mathcal{U}_{\alpha}\\i\in U}}(-1)^{\Lambda_{\alpha}(U)}l_{U\setminus\{i\}\cup\{n-m\}}
        = -\sum_{\substack{U\in\mathcal{U}_{\alpha}\\i\in U}}(-1)^{\Lambda_{\alpha}(U)}l_{U\setminus\{i\}\cup\{n-m\}}
        = -x_{n-m}\cdot L_{\alpha},
    \end{gather*}
    which implies that $x_{n-m}\cdot L_{\alpha}=0$.
\end{proof}

\begin{rem}
The proof of \autoref{lem:action_dunkl_ray} (as well as the proof of \autoref{lem:action_dunkl_right} below) uses a well-known argument from algebraic combinatorics. As in \cite[Section 2.6]{stanley}, we show the vanishing of some terms by constructing sign-reversing involutions on the summands.
\end{rem}

\begin{lem}\label{lem:action_dunkl_left}
    Suppose that $n-m$ is the left endpoint of a cup in $C(\alpha)$. Then 
    \begin{gather*}
        x_{n-m}\cdot L_{\alpha} = (m+1) L_{\alpha}.
    \end{gather*}
\end{lem}

\begin{proof}
    Since $n-m$ is the left endpoint of a cup in $C(\alpha)$, the corresponding right endpoint is on the right of the cut line. Therefore $n-m\in U$ for all $U\in \mathcal{U}_{\alpha}$ and the definition of the action of the degenerate affine Hecke algebra in terms of Dunkl operators gives
    \begin{gather*}
        x_{n-m}\cdot l_{U} = \left(\xi_d-\lvert\{1 \leq i < n-m\mid i\not\in U\}\rvert\right) l_U = (n+1-d - (n-m-d))l_U = (m\!+\!1)l_U
    \end{gather*}
    for all $U\in\mathcal{U}_{\alpha}$. This implies that $x_{n-m}\cdot L_{\alpha} = (m+1)L_{\alpha}$.
\end{proof}

In order to state the last case, we need to introduce some notation. Suppose that $n-m$ is the right endpoint of a cup in $C(\alpha)$. Denote by $a_1< \cdots < a_r$ the left endpoints of cups in $C(\alpha)$ with right endpoints among the last $m$ points, \emph{i.e.} right of the cut line. We also let $a_{r+1}$ be the left endpoint of the cup of $C(\alpha)$ with right endpoint $n-m$.

Therefore, the weight $\alpha$ satisfies the following: 
\begin{itemize}
    \item $\alpha_{j} = \up$ for $n-m+1 \leq j \leq n-m+r$ and $\alpha_{j} = \down$ for $n-m+r+1 \leq j \leq n$,
    \item $\alpha_{j} = \down$ for $j\in \{a_1,\ldots,a_{r+1}\}$ and $\alpha_{n-m} = \up$.
\end{itemize}

For $1\leq i \leq r$ define a $\Delta$-weight $\alpha^{i}$ of type $(n-k,k,m)$ by:
\begin{itemize}
    \item $\alpha^{i}_{n-m} = \down$ and $\alpha^{i}_{a_{i+1}} = \up$,
    \item $\alpha^{i}_j = \alpha_j$ for $j\neq a_{i+1}, n-m$.
\end{itemize}
In terms of cup diagrams, the diagram $C(\alpha^i)$ is obtained from $C(\alpha)$ by moving the cup joining $a_{r+1}$ and $n-m$ to a cup joining $a_{i}$ and $a_{i+1}$, and by moving the left endpoints of the $r-i+1$ rightmost cups crossing the cut line.

If the number of right endpoints of cups crossing the cut line in $C(\alpha)$ is smaller than $m$, that is $r<m$, we also define a weight $\alpha^{0}$ by:
\begin{itemize}
    \item $\alpha^{0}_{n-m} = \down$ and $\alpha^{0}_{n-m+r+1} = \up$,
    \item $\alpha^{0}_j = \alpha_j$ for $j\neq a_{1}, n-m$.
\end{itemize}
In terms of cup diagram, all the points $a_2,\ldots,a_{r+1}$ and $n-m$ are left endpoints of a cup and $a_1$ is the endpoint of a ray.

\begin{exe}
    Let us take $\alpha=\down\down\up\down\down\up\vert\!\up\up\down$. The associated cup diagram is then
    \begin{gather*}
            C(\alpha) = 
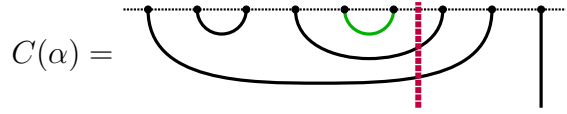
\xy (0,0)*{
            \begin{tikzpicture}[scale=0.65]
            \draw[thick,densely dotted] (-.5,1) to (8.5,1);
            \draw[very thick] (1,1) to[out=-90,in=180] (1.5,.5) to[out=0,in=-90] (2,1);
            \draw[mygreen,very thick] (4,1) to[out=-90,in=180] (4.5,.5) to[out=0,in=-90] (5,1);
            \draw[very thick] (3,1) to[out=-90,in=180] (4.5,0) to[out=0,in=-90] (6,1);
            \draw[very thick] (0,1) to[out=-90,in=180] (3.5,-0.5) to[out=0,in=-90] (7,1);
            \draw[very thick] (8,1) -- (8,-1);
            \draw[black,fill=black] (0,1) circle (2.0pt);
            \draw[black,fill=black] (1,1) circle (2.0pt);
            \draw[black,fill=black] (2,1) circle (2.0pt);
            \draw[black,fill=black] (3,1) circle (2.0pt);
            \draw[black,fill=black] (4,1) circle (2.0pt);
            \draw[black,fill=black] (5,1) circle (2.0pt);
            \draw[black,fill=black] (6,1) circle (2.0pt);
            \draw[black,fill=black] (7,1) circle (2.0pt);
            \draw[black,fill=black] (8,1) circle (2.0pt);
            \draw[cutline] (5.5,-1) to (5.5,1.15);
            \end{tikzpicture}
            }\endxy 
        \end{gather*}

        We have $r=2$ and $a_1=1$, $a_2=4$ and $a_3=5$. Therefore, the weights $\alpha^{2}$, $\alpha^{1}$ and $\alpha^{0}$ and their respective associated cup diagram are given by
        \begin{align*}
            \alpha^{2}&=\down\down\up\down\up\down\vert\!\up\up\down ,& C(\alpha^{2})&=
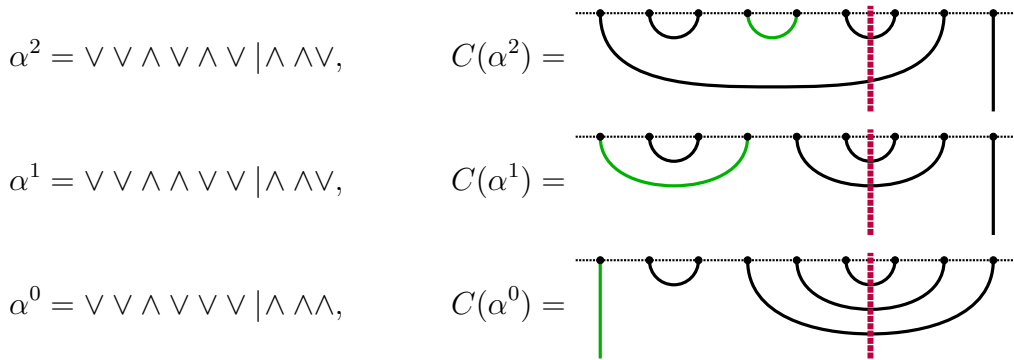
\xy (0,0)*{
            \begin{tikzpicture}[scale=0.65]
            \draw[thick,densely dotted] (-.5,1) to (8.5,1);
            \draw[very thick] (1,1) to[out=-90,in=180] (1.5,.5) to[out=0,in=-90] (2,1);
            \draw[mygreen,very thick] (3,1) to[out=-90,in=180] (3.5,.5) to[out=0,in=-90] (4,1);
            \draw[very thick] (5,1) to[out=-90,in=180] (5.5,.5) to[out=0,in=-90] (6,1);
            \draw[very thick] (0,1) to[out=-90,in=180] (3.5,-0.5) to[out=0,in=-90] (7,1);
            \draw[very thick] (8,1) -- (8,-1);
            \draw[black,fill=black] (0,1) circle (2.0pt);
            \draw[black,fill=black] (1,1) circle (2.0pt);
            \draw[black,fill=black] (2,1) circle (2.0pt);
            \draw[black,fill=black] (3,1) circle (2.0pt);
            \draw[black,fill=black] (4,1) circle (2.0pt);
            \draw[black,fill=black] (5,1) circle (2.0pt);
            \draw[black,fill=black] (6,1) circle (2.0pt);
            \draw[black,fill=black] (7,1) circle (2.0pt);
            \draw[black,fill=black] (8,1) circle (2.0pt);
            \draw[cutline] (5.5,-1) to (5.5,1.15);
            \end{tikzpicture}
            }\endxy \\
            \alpha^{1}&=\down\down\up\up\down\down\vert\!\up\up\down ,& C(\alpha^{1})&=\xy (0,0)*{
            \begin{tikzpicture}[scale=0.65]
            \draw[thick,densely dotted] (-.5,1) to (8.5,1);
            \draw[very thick] (1,1) to[out=-90,in=180] (1.5,.5) to[out=0,in=-90] (2,1);
            \draw[very thick] (5,1) to[out=-90,in=180] (5.5,.5) to[out=0,in=-90] (6,1);
            \draw[very thick] (4,1) to[out=-90,in=180] (5.5,0) to[out=0,in=-90] (7,1);
            \draw[mygreen,very thick] (0,1) to[out=-90,in=180] (1.5,0) to[out=0,in=-90] (3,1);
            \draw[very thick] (8,1) -- (8,-1);
            \draw[black,fill=black] (0,1) circle (2.0pt);
            \draw[black,fill=black] (1,1) circle (2.0pt);
            \draw[black,fill=black] (2,1) circle (2.0pt);
            \draw[black,fill=black] (3,1) circle (2.0pt);
            \draw[black,fill=black] (4,1) circle (2.0pt);
            \draw[black,fill=black] (5,1) circle (2.0pt);
            \draw[black,fill=black] (6,1) circle (2.0pt);
            \draw[black,fill=black] (7,1) circle (2.0pt);
            \draw[black,fill=black] (8,1) circle (2.0pt);
            \draw[cutline] (5.5,-1) to (5.5,1.15);
            \end{tikzpicture}
            }\endxy \\
            \alpha^{0}&=\down\down\up\down\down\down\vert\!\up\up\up ,& C(\alpha^{0})&=\xy (0,0)*{
            \begin{tikzpicture}[scale=0.65]
            \draw[thick,densely dotted] (-.5,1) to (8.5,1);
            \draw[very thick] (1,1) to[out=-90,in=180] (1.5,.5) to[out=0,in=-90] (2,1);
            \draw[very thick] (3,1) to[out=-90,in=180] (5.5,-.5) to[out=0,in=-90] (8,1);
            \draw[very thick] (4,1) to[out=-90,in=180] (5.5,0) to[out=0,in=-90] (7,1);
            \draw[very thick] (5,1) to[out=-90,in=180] (5.5,0.5) to[out=0,in=-90] (6,1);
            \draw[mygreen,very thick] (0,1) -- (0,-1);
            \draw[black,fill=black] (0,1) circle (2.0pt);
            \draw[black,fill=black] (1,1) circle (2.0pt);
            \draw[black,fill=black] (2,1) circle (2.0pt);
            \draw[black,fill=black] (3,1) circle (2.0pt);
            \draw[black,fill=black] (4,1) circle (2.0pt);
            \draw[black,fill=black] (5,1) circle (2.0pt);
            \draw[black,fill=black] (6,1) circle (2.0pt);
            \draw[black,fill=black] (7,1) circle (2.0pt);
            \draw[black,fill=black] (8,1) circle (2.0pt);
            \draw[cutline] (5.5,-1) to (5.5,1.15);
            \end{tikzpicture}
            }\endxy \\
        \end{align*}
    We have drawn in green the cup that \say{moves along the diagrams when going from $\alpha$ to $\alpha^i$}.
\end{exe}

\begin{lem}\label{lem:action_dunkl_right}
    Suppose that $n-m$ is the right endpoint of a cup in $C(\alpha)$, and keep the notations from above. Then 
    \begin{gather}\label{eq:action_dunkl_right}
        x_{n-m} \cdot L_{\alpha} = -\sum_{i=0}^r(m-r+i)L_{\alpha^{i}}.
    \end{gather}
\end{lem}

\begin{rem}
    In \eqref{eq:action_dunkl_right}, the term $L_{\alpha^{0}}$ is not defined if $m=r$. Nonetheless, in the case $m=r$, this term appears with a coefficient $0$ and it is unnecessary to distinguish the cases $m=r$ and $m\neq r$.
\end{rem}

\begin{exe}
    In the above example, we obtain $x_{6}\cdot L_{\alpha} = -3L_{\alpha^2}-2L_{\alpha^1}-L_{\alpha^0}$.
\end{exe}

\begin{proof}[Proof of \autoref{lem:action_dunkl_right}]
    First, we compute the action of $x_{n-m}$ on $L_{\alpha}$ following the definition in terms of Dunkl operators:
    \begin{gather*}
        x_{n-m}\cdot L_{\alpha} = \sum_{\substack{U\in\mathcal{U}_{\alpha}\\ n-m\in U}}(-1)^{\Lambda_{\alpha}(U)}(m+1)l_U 
        + \sum_{\substack{U\in\mathcal{U}_{\alpha}\\ n-m\not\in U}}(-1)^{\Lambda_{\alpha}(U)}\sum_{i\in U} l_{U\setminus\{i\}\cup\{n-m\}},
    \end{gather*}
    where we have used that $\xi_d=n+1-d$ and $\lvert\{1\leq i < n-m\mid i\not\in U\}\rvert = n-m-d$ for $U\in \mathcal{U}_{\alpha}$ with $n-m\not\in U$.

    In order to compare with the right hand side of the equality to prove, we introduce some subsets of $\{1,\ldots,n-m\}$. For $1\leq i \leq r$ we define
    \begin{gather*}
        \mathcal{U}_i = \{U \in \mathcal{U}_{\alpha^{i}}\mid a_i \in U\} = \{U \in \mathcal{U}_{\alpha^{i}}\mid a_{i+1} \not\in U\},
    \end{gather*}
    and also $\mathcal{U}_{0} = \mathcal{U}_{\alpha^{0}}$ (if $\alpha^{0}$ is not defined, then we set $\mathcal{U}_{0} = \emptyset$). It is then straightforward to check that, for all $1\leq i \leq r$, we have $\mathcal{U}_{\alpha^{i}} = \mathcal{U}_{i-1} \cup \mathcal{U}_i$, and that these unions are disjoint. Moreover, an element $U\in \mathcal{U}_i$ is as well in $\mathcal{U}_{\alpha^{i+1}}$ and we have $\Lambda_{\alpha^{i+1}}(U) = \Lambda_{\alpha^{i}}(U) + 1$. Therefore
    \begin{align*}
        \sum_{i=0}^{r}(m-r+i)&L_{\alpha^{i}} 
        = \sum_{i=1}^{r}(m-r+i)\sum_{U\in \mathcal{U}_{i-1}}(-1)^{\Lambda_{\alpha^{i}}(U)}l_U + \sum_{i=0}^{r}(m-r+i)\sum_{U\in \mathcal{U}_{i}}(-1)^{\Lambda_{\alpha^{i}}(U)}l_U\\
        =& -\sum_{i=0}^{r-1}(m-r+i+1)\sum_{U\in \mathcal{U}_{i}}(-1)^{\Lambda_{\alpha^{i}}(U)}l_U + \sum_{i=0}^{r}(m-r+i)\sum_{U\in \mathcal{U}_{i}}(-1)^{\Lambda_{\alpha^{i}}(U)}l_U\\
        =& m\sum_{U\in \mathcal{U}_{r}}(-1)^{\Lambda_{\alpha^{r}}(U)}l_U - \sum_{i=0}^{r-1}\sum_{U\in \mathcal{U}_{i}}(-1)^{\Lambda_{\alpha^{i}}(U)}l_U \\
        =& (m+1)\sum_{U\in \mathcal{U}_{r}}(-1)^{\Lambda_{\alpha^{r}}(U)}l_U - \sum_{i=0}^{r}\sum_{U\in \mathcal{U}_{i}}(-1)^{\Lambda_{\alpha^{i}}(U)}l_U.
    \end{align*}

    But we have $\mathcal{U}_{r} = \{U\in \mathcal{U}_{\alpha} \mid n-m\in U\}$ and if $U\in \mathcal{U}_{r}$ then $\Lambda_{\alpha}(U) = \Lambda_{\alpha^{r}}(U)+1$. It remains to prove the following equality:
    \begin{gather*}
        \sum_{\substack{U\in\mathcal{U}_{\alpha}\\ n-m\not\in U}}(-1)^{\Lambda_{\alpha}(U)}\sum_{i\in U} l_{U\setminus\{i\}\cup\{n-m\}} = \sum_{i=0}^{r}\sum_{U\in \mathcal{U}_{i}}(-1)^{\Lambda_{\alpha^{i}}(U)}l_U.
    \end{gather*}

    We simplify the left hand side as follows. Suppose that $i$ and $j$ are connected by a cup in $C(\alpha)$. Then we have a bijection 
    \begin{gather*}
        \{U\in \mathcal{U}_{\alpha}\mid n-m \not\in U\text{ and } i\in U\} \to \{U\in \mathcal{U}_{\alpha}\mid n-m \not\in U\text{ and } j\in U\}
    \end{gather*}
    given by $U\mapsto U\setminus\{i\}\cup\{j\}$, and $\Lambda_{\alpha}(U\setminus\{i\}\cup\{j\}) = \Lambda_{\alpha}(U) \pm 1$. Therefore
    \begin{gather*}
        \sum_{\substack{U\in\mathcal{U}_{\alpha}\\ n-m\not\in U}}(-1)^{\Lambda_{\alpha}(U)}\sum_{i\in U} l_{U\setminus\{i\}\cup\{n-m\}}
        =
        \sum_{\substack{U\in\mathcal{U}_{\alpha}\\ n-m\not\in U}}(-1)^{\Lambda_{\alpha}(U)}\sum_{\substack{i=1\\a_i\in U}}^{r+1} l_{U\setminus\{a_i\}\cup\{n-m\}}, 
    \end{gather*}
    since we can cancel the other terms by pairs that carry opposite signs. Finally, for all $1\leq i \leq r+1$, we have a bijection between $\{U\in\mathcal{U}_{\alpha}\mid n-m\not\in U\text{ and } a_i\in U\}$ and $\mathcal{U}_{i-1}$ given by $U\mapsto U\setminus\{a_i\}\cup\{n-m\}$ which moreover satisfies $\Lambda_{\alpha}(U) = \Lambda_{\alpha^{i}}(U\setminus\{a_i\}\cup\{n-m\})$. Therefore
    \begin{gather*}
        \sum_{\substack{U\in\mathcal{U}_{\alpha}\\ n-m\not\in U}}(-1)^{\Lambda_{\alpha}(U)}\sum_{\substack{i=1\\a_i\in U}}^{r+1} l_{U\setminus\{a_i\}\cup\{n-m\}}
        =
        \sum_{i=0}^{r}\sum_{U\in \mathcal{U}_{i}}(-1)^{\Lambda_{\alpha^{i}}(U)}l_U
    \end{gather*}
    and the proof is complete.
\end{proof}

\begin{proof}[Proof of \autoref{thm:stability_degenerate}]
The theorem follows from \autoref{prop:action_symmetric}, \autoref{lem:action_dunkl_ray}, \autoref{lem:action_dunkl_left} and \autoref{lem:action_dunkl_right}, together with the fact that the degenerate affine Hecke algebra is generated by the symmetric group and by $x_{n-m}$.
\end{proof}

\subsection{Comparison with a tensor space}

The goal is to give a proof of the irreducibility of the action of $H_{n-m}$ on the cohomology of the $\Delta$-Springer variety, using Schur--Weyl duality. We consider the Lie algebra $\mathfrak{gl}_2$ with basis given by $(h_1,h_2,e,f)$ where
\[
    h_1=\begin{pmatrix}1 & 0 \\ 0 & 0 \end{pmatrix},\ 
    h_2=\begin{pmatrix}0 & 0 \\ 0 & 1 \end{pmatrix},\ 
    e=\begin{pmatrix}0 & 1 \\ 0 & 0 \end{pmatrix},\ f=\begin{pmatrix}0 & 0 \\ 1 & 0 \end{pmatrix}.
\]
We denote by $\Omega$ its Casimir element, that is $\Omega = h_1\otimes h_1 + h_2\otimes h_2 + e\otimes f + f\otimes e$. We denote by $\Omega_{i,j}\in \mathfrak{gl}_2^{\otimes n}$, for $1\leq i < j \leq n$, the Casimir element for the $i$th and $j$th factor.

We also consider the natural representation $V$ with basis $(v_{+},v_{-})$ with the following $\mathfrak{gl}_2$-action:
\begin{align*}
    h_1\cdot v_+ &= v_+, & h_2\cdot v_+ &= 0, & e\cdot v_+ &= 0, & f \cdot v_+ &= v_-,\\
    h_1\cdot v_- &= 0, & h_2\cdot v_- &= v_-, & e\cdot v_- &= v_+, & f \cdot v_- &= 0.
\end{align*}
The $m$-th symmetric power $S^{m}V$ of the natural representation has $(v_0,\ldots,v_m)$ as a basis and the following $\mathfrak{gl}_2$-action:
\begin{align*}
    h_1\cdot v_i &= (m-i)v_i, & h_2\cdot v_i &=iv_i & e\cdot v_i &= iv_{i-1}, & f\cdot v_i &= (m-i)v_{i+1}.
\end{align*}

The main player will be the tensor space $V^{\otimes(n-m)}\otimes S^{m}V$, which is naturally a $\mathfrak{gl}_2$-module. The following theorem has been proved in \cite[Proposition 7.1]{suz98}.

\begin{thm}
    \begin{enumerate}
        \item The mapping $s_i\mapsto \Omega_{i,i+1}$ and $x_i \mapsto -\sum_{i<j\leq n-m+1}\Omega_{i,j}+m\id$ defines an action of $H_{n-m}$ on $V^{\otimes(n-m)}\otimes S^mV$, which commutes with the action of $\mathfrak{gl}_2$.
        \item The maps $H_{n-m}\to\End_{\mathfrak{gl}_2}(V^{\otimes(n-m)}\otimes S^mV)$ and $U(\mathfrak{gl}_2)\to\End_{H_{n-m}}(V^{\otimes(n-m)}\otimes S^mV)$ are surjective.
        \item As an $(U(\mathfrak{gl}_2),H_{n-m})$-bimodule, we have
        \[
        V^{\otimes(n-m)}\otimes S^mV \simeq \bigoplus_{d=0}^{\min(n-m,\lfloor n/2 \rfloor)}V_{(n-d,d)}\boxtimes L_d,
        \]
        Where $V_{(n-d,d)}$ is the simple highest $\mathfrak{gl}_2$-weight module of highest weight $(n-d,d)$ and $L_d$ is a simple $H_{n-m}$-module.
    \end{enumerate}
\end{thm}

This theorem provides a construction of certain simple $H_{n-m}$-modules, which will be compared to the modules obtained via the cohomology of the $\Delta$-Springer variety.

\begin{rem}
    In \cite{suz98}, the action of $x_i$ is slightly different. We recover ours after applying the automorphism $\sigma_i\mapsto\sigma_{n-m-i}$, $x_i\mapsto -x_{n-m+1-i}$ of $H_{n-m}$ and shifting the action of all the $x_i$ by a common constant.

    The simple module $L_{d}$ is a calibrated simple module, in the sense of \cite{ram03}, and is then indexed by a placed skew shape. It can be checked that the corresponding skew partition is $(n-d,d)/(m)$.
\end{rem}

First, to each $\ba\in\mathbb{B}_{n-d,d,m}$ we associate a highest weight vector $p_{\ba}$ of weight $(n-d,d)$ in $V^{\otimes(n-m)}\otimes S^{m}V$. We convert each cup left of the cut line into the vector $v_{-}\otimes v_{+} - v_{+}\otimes v_{-}$, each ray left of the cut line into $v_{+}$, and the remaining $r$ cups with endpoints on both sides of the cut line into the vector
\[
    w_r=\sum_{I\subset\{1,\ldots,r\}}(-1)^{\lvert I \rvert}v_{I}\otimes v_{\lvert I \rvert},
\]
where $v_{I}\in V^{\otimes r}$ has its $i$-th tensorand equal to $v_{+}$ if $i\in I$ and to $v_{-}$ otherwise. We then place these vectors in the appropriate position in $V^{\otimes(n-m)}\otimes S^{m}V$: the $n-m$ endpoints left of the cut line correspond to the various tensorands $V$ and the $m$ endpoints right of the cut line correspond to the tensorand $S^mV$. In this process, we completely ignore the rays right of the cut line.

\begin{exe}
The vector corresponding to the diagram
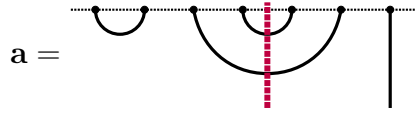
\begin{equation*}
\ba = \xy (0,0)*{
\begin{tikzpicture}[scale=0.65]
\draw[thick,densely dotted] (-.5,1) to (6.5,1);
\draw[black,fill=black] (0,1) circle (2.0pt);
\draw[black,fill=black] (1,1) circle (2.0pt);
\draw[black,fill=black] (2,1) circle (2.0pt);
\draw[black,fill=black] (3,1) circle (2.0pt);
\draw[black,fill=black] (4,1) circle (2.0pt);
\draw[black,fill=black] (5,1) circle (2.0pt);
\draw[black,fill=black] (6,1) circle (2.0pt);
\draw[very thick] (0,1) to[out=-90,in=180] (.5,.5) to[out=0,in=-90] (1,1);
\draw[very thick] (2,1) to[out=-80,in=180] (3.5,-.3) to[out=0,in=-100] (5,1);
\draw[very thick] (3,1) to[out=-90,in=180] (3.5,.5) to[out=0,in=-90] (4,1);
\draw[very thick] (6,1) -- (6,-1);
\draw[cutline] (3.5,-1) to (3.5,1.15);
\end{tikzpicture}
}\endxy
\end{equation*}
is
\begin{align*}
    (v_-\otimes v_+ - v_+\otimes v_-)\otimes (v_-\otimes v_-\otimes v_0
    -v_+\otimes v_-\otimes v_1
    -v_-\otimes v_+\otimes v_1
    +v_+\otimes v_+\otimes v_2).
\end{align*}
\end{exe}

\begin{lem}
    Let $\ba\in\mathbb{B}_{n-d,d,m}$. Then $p_{\ba}$ is a highest weight vector of weight $(n-d,d)$.
\end{lem}

\begin{proof}
    Since permuting the first $n-m$ tensorands of $V^{\otimes(n-m)}\otimes S^mV$ commutes with the action of $\mathfrak{gl}_2$, it suffices to check that $v_-\otimes v_+-v_+\otimes v_-$, $v_+$ and $w_r$ are highest weight vectors, which is immediate. The weights of $v_-\otimes v_+-v_+\otimes v_-$, $v_+$ and $w_r$ are respectively $(1,1)$, $(1,0)$ and $(m,r)$, the weight of $p_{\ba}$ is then $(n-d,d)$.
\end{proof}

We also notice that $p_{\ba} = v_{\alpha(\ba)} + \text{higher terms}$, where $\alpha(\ba)$ is the $\Delta$-weight obtained by orienting the cups in $\ba$ counterclockwise and the rays upward, and $v_{\alpha(\ba)}$ is obtained by replacing $\down$ by $v_{-}$ and $\up$ by $v_{+}$ and then tensoring by $v_{0}$. Here, the higher terms are with respect to  $v_{\varepsilon_{1}}\otimes\cdots\otimes v_{\varepsilon_{n-m}}\otimes v_{i} < v_{\varepsilon'_{1}}\otimes\cdots\otimes v_{\varepsilon'_{n-m}}\otimes v_{i'}$ if, for the first index such that $\varepsilon_{i}\neq \varepsilon'_{i}$, we have $\varepsilon_i < \varepsilon'_i$, with the convention $-<+$. Therefore, the family $\{p_{\ba} \mid \ba\in \mathbb{B}_{n-d,d,m}\}$ is free.

\begin{lem}
    The set $\{p_{\ba} \mid \ba\in \mathbb{B}_{n-d,d,m}\}$ is a basis of the highest weight vectors of weight $(n-d,d)$ in $V^{\otimes(n-m)}\otimes S^{m}V$.
\end{lem}

\begin{proof}
    Since the family is linearly independent, it remains to show that we have the correct numbers of such vectors. But, using the branching rules for $\mathfrak{gl}_{2}$-modules, we see that the number of summands of $V^{\otimes(n-m)}\otimes S^{m}V$ isomorphic to $V_{(n-d,d)}$ is the number of standard tableaux of skew shape $(n-d,d)/(m)$, which is also the cardinality of $\mathbb{B}_{n-d,d,m}$.
\end{proof}

As a corollary, we have a basis of the $H_{n-m}$-module $L_{d}$, and we now compute the action of the generators of the degenerate affine Hecke algebra to compare with the representation obtained on the cohomology of the $\Delta$-Springer variety. 

\begin{prop}
    Let $\ba\in\mathbb{B}_{n-d,d,m}$. The action of $s_i$ on $p_{\ba}$ is obtained by stacking the diagram
    \begin{gather*}
        \begin{tikzpicture}
            \draw[thick] (0,0) -- +(0,-1);
            \draw[thick] (.75,0) -- +(0,-1);
            \draw[thick] (1.25,0) -- +(.5,-1);
            \draw[thick] (1.75,0) -- +(-.5,-1);
            \draw[thick] (2.25,0) -- +(0,-1);
            \draw[thick] (3,0) -- +(0,-1);
            \node at (.4,-.5) {$\cdots$};
            \node at (2.65,-.5) {$\cdots$};
            \node[above] at (1.20,0.03) {\tiny{$i$}};
            \node[above] at (1.75,0) {\tiny{$i+1$}};
        \end{tikzpicture}
    \end{gather*}
    on top of $\ba$ and using the same relations as in \autoref{prop:action_symmetric}.
\end{prop}

\begin{proof}
    This follows from standard calculations. For example, the vector $w_r$ is clearly invariant by permutation of any two consecutive tensorands among the first $r$, which corresponds to the last point of \autoref{prop:action_symmetric}.
\end{proof}

It remains to compute the action of the polynomial generator of the degenerate affine Hecke algebra. We compute the action of $x_{n-m}$ by distinguishing three cases.

\begin{prop}
    \label{prop:action_quantum_diagrams}
    Let $\ba\in \mathbb{B}_{n-d,d,m}$. We have:
    \[
        x_{n-m}\cdot p_{\ba} = 
        \begin{cases}
            0 & \text{if } n-m \text{ is the endpoint of a ray},\\
            (m+1)p_{\ba} & \text{if } n-m \text{ is the left endpoint of a cup},\\
            - \sum_{i=0}^{r}(m-r+i)p_{\ba^i} &\text{if } n-m \text{ is the right endpoint of a cup},
        \end{cases}
    \]
    where, in the third case, the integer $r$ and the diagrams $\ba^{i}$ are defined similarly as in the previous section.
\end{prop}

\begin{proof}
    We first suppose that $n-m$ is the endpoint of a ray in $\ba$. Therefore, there are only rays right of the cut line, and $p_{\ba} = v\otimes v_+\otimes v_0$ for some $v\in V^{\otimes(n-m-1)}$. Since $\Omega\cdot (v_+\otimes v_0) = mv_+\otimes v_0$, we find that $x_{n-m}\cdot p_{\ba} = 0$.
    
    Now, we suppose that $n-m$ is the left endpoint of a cup. The corresponding right endpoint is then right of the cut line. Using the action of the symmetric group, it then suffices to compute the action of $\Omega_{r,r+1}$ on the vector $w_r$. Given $I\subset\{1,\ldots,r\}$, we have 
    \[
    \Omega_{r,r+1}\cdot(v_I\otimes v_{\lvert I\rvert}) = 
    \begin{cases}
        (m-\lvert I \rvert) v_I\otimes v_{\lvert I \rvert} + \lvert I \rvert v_{I\backslash\{n-m\}}\otimes v_{\lvert I \rvert -1} & \text{if } n-m \in I,\\
        \lvert I \rvert v_I\otimes v_{\lvert I \rvert} + (m-\lvert I \rvert) v_{I\cup\{n-m\}}\otimes v_{\lvert I \rvert +1}& \text{if } n-m \not\in I.
    \end{cases}
    \]
    This implies that $\Omega_{r,r+1}\cdot w_r = -w_r$ and that $x_{n-m}\cdot p_{\ba} = (m+1)p_{\ba}$.
    
    Finally, suppose that $n-m$ is the right endpoint of a cup. Once again, using the action of the symmetric group, it suffices to compute the action of $\Omega_{r+2,r+3}$ on the vector 
    \[
    \sum_{I\subset\{1,\ldots,r\}} (-1)^{\lvert I \rvert} v_{I}\otimes(v_-\otimes v_+ - v_+\otimes v_-) \otimes v_{\lvert I \rvert}. 
    \]
    For $1\leq i \leq r$, denote by $w_{r,i,i+1}$, the vector in $V^{\otimes(r+2)}\otimes S^mV$ given by inserting into $w_r$ the vector $v_-\otimes v_+ - v_+\otimes v_i$ at the tensorands $i$ and $i+1$. We have
    \begin{align*}
        \Omega_{r+2,r+3}\cdot w_{r,r+1,r+2} = \sum_{I\subset\{1,\ldots,r\}}(-1)^{\lvert I \rvert} v_{I}\otimes&\big(v_-\otimes ((m-\lvert I\rvert)v_+\otimes v_{\lvert I\rvert} + \lvert I \rvert v_{-}\otimes v_{\lvert I \rvert -1})\\
        &-v_+\otimes(\lvert I \rvert v_-\otimes v_{\lvert I \rvert} + (m-\lvert I\rvert)v_+\otimes v_{\lvert I \rvert +1})\big).
    \end{align*}
    We then find that $\Omega_{r+2,r+3}\cdot w_{r,r+1,r+2}+w_{r,r+1,r+2}$ is a linear combination of terms of the form $v_{J}\otimes v_{\lvert J \rvert -1}$, with $J\subset\{1,\ldots,r+2\}$, with coefficients $(-1)^{\lvert J\rvert -1}(m-\lvert J \rvert +2)$ if $r+2\in J$ and $(-1)^{\lvert J\rvert}\lvert J \rvert$ if $r+2\not\in J$. We fix $J\subset\{1,\ldots,r+2\}$ and compute the coefficient of $v_{J}\otimes v_{\lvert J \rvert}$ in the sum
    \begin{gather}
    \label{eq:sum}
        \sum_{i=1}^{r+1}(m-r+i)w_{r,i,i+1} + (m-r)v_+\otimes w_{r+1}.
    \end{gather}
    We will distinguish four cases, whether $1$ and $r+2$ are in $J$ or not. We will treat only the case $1,r+2\in J$, the other cases being similar. The set $J$ correspond to a word $+^{\alpha_1}-^{\beta_1}\cdots-^{\beta_k}+^{\alpha_{k+1}}$, with $\alpha_i,\beta_i>0$. Then, for each subword $-+$ in the position $i,i+1$, $v_{J}\otimes v_{\lvert J \rvert -1}$ appears in $w_{r,i,i+1}$ with a coefficient $(-1)^{\lvert J \rvert -1}$ and for each subword $+-$ in the position $i,i+1$, $v_{J}\otimes v_{\lvert J \rvert -1}$ appears in $w_{r,i,i+1}$ with a coefficient $-(-1)^{\lvert J \rvert-1}$. Therefore, the total coefficient of $v_{J}\otimes v_{\lvert J \rvert -1}$ in the sum \eqref{eq:sum} is
    \begin{align*}
        &(-1)^{\lvert J \rvert -1}\left(\sum_{j=1}^k(m-r+\sum_{s=1}^{j}(\alpha_s+\beta_s))-\sum_{j=1}^k(m-r+\sum_{s=1}^{j-1}(\alpha_s+\beta_s)+\alpha_j)+(m-r)\right)\\
        &=(-1)^{\lvert J \rvert -1}(m-r+\sum_{j=1}^{k}\beta_k).
    \end{align*}
    But $\sum_{j=1}^{k}\beta_k = (r+2)-\lvert J \rvert$ and we find that the coefficient is $(-1)^{\lvert J \rvert -1}(m-\lvert J \rvert +2)$ as expected. We finally obtain that
    \begin{align*}
        \Omega_{r+2,r+3}\cdot w_{r,r+1,r+2} 
        &= -w_{r,r+1,r+2} + \sum_{i=1}^{r+1}(m-r+i)w_{r,i,i+1} + (m-r)v_+\otimes w_{r+1}\\
        &= mw_{r,r+1,r+2} + \sum_{i=1}^{r}(m-r+i)w_{r,i,i+1} + (m-r)v_+\otimes w_{r+1},
    \end{align*}
    which concludes the proof.
\end{proof}

\begin{thm}\label{cor:irred_H_action_Delta}
    The $H_{n-m}$-module $H_{2d(\mathcal{S}^{\Delta}_{(n-k,k),m})}$ is irreducible.
\end{thm}

\begin{proof}
    The map $L_{\alpha}\mapsto p_{\ba(\alpha)}$ is an isomorphism of vector spaces since it sends a basis of $H_{2d(\mathcal{S}^{\Delta}_{(n-k,k),m})}$ onto a basis of $L_{d}$. Furthermore, the comparison of the action of $H_{n-m}$ on both sides shows that this isomorphism is $H_{n-m}$-equivariant, see \autoref{lem:action_dunkl_ray}, \autoref{lem:action_dunkl_left} and \autoref{lem:action_dunkl_right} for one side, and \autoref{prop:action_quantum_diagrams} for the other. Since $L_{d}$ is irreducible, so is $H_{2d}(\dspr_{(n-k,k),m})$.
\end{proof}

\subsection{Action in the extremal cases}

\subsubsection{Springer fiber of type \texorpdfstring{$A$}{A}}

In this section, we assume that $m=0$. In this extremal case, the $\Delta$-Springer variety is isomorphic to the Springer fiber associated with the two-row partition $(n-k,k)$. There is a well-known action of the symmetric group $S_n$ on the (co)homology of the Springer fiber, which has a skein theoretic interpretation \cite{russell}. We recover this action if we restrict the action of the degenerate affine Hecke algebra $H_n$ to the symmetric group $S_n$.

\begin{prop}
    Suppose that $m=0$. Then the generator $x_n$ acts by $0$. 
\end{prop}

\begin{proof}
    Since $m=0$, there are no points on the right of the cut line, and $n$ cannot be the left endpoint of a cup. The result then follows from \autoref{lem:action_dunkl_ray} and \autoref{lem:action_dunkl_right}.
\end{proof}

The action of the generator $x_i$ is then given by the Jucys--Murphy element $\sum_{j=i+1}^n (i\ j)$. The action of $H_n$ is entirely recovered from the action of the (group algebra of the) symmetric group $S_n$.

\subsubsection{Exotic Springer fiber}

In the extremal case $m=k$, the $\Delta$-Springer variety is isomorphic to the exotic Springer fiber associated with the one-row bipartition $(n-2k,k)$. There is a well-known action of the Weyl group $W_{n-k}$ of type $C_{n-k}$ on the (co)homology of the exotic Springer fiber, which can be easily described~\cite{SW22}. 

Recall that in~\autoref{S:Haction} we have constructed an action of the degenerate affine Hecke algebra $H_{n-m}$ on the homology of the $\Delta$-Springer variety $\dspr_{(n-k,k),m}$. We now describe this action using $W_{n-m}$,  which we consider as the group of signed permutations on the set $\{-(n-m),\ldots,-1,1,\ldots,n-m\}$. 

For $1\leq i < n-m$, we denote by $s_i$ the permutation $(i\ i+1)(-i\ -(i+1))$ and by $s_i'$ the permutation $(i\ -i)$. The group $W_{n-m}$ acts on $H_{2d}((\mathbb{S}^2)^{n-m})$: the element $s_i$ acts by usual permutations and $s'_i$ acts on a line diagram $l_U\in H_{2d}((\mathbb{S}^2)^{n-m})$ by
\[
s'_i \cdot l_U = 
\begin{cases}
l_U & \text{if }i\not\in U,\\
-l_U & \text{if }i\in U.
\end{cases}
\]

For $1\leq i \leq n-m$, we define the Jucys--Murphy elements $J_i$ and $\tilde{J}_i$ by
\[
J_i = \sum_{j=1}^{i-1} (j\,\ i)
\quad\text{and}\quad
\tilde{J}_i = \sum_{j=i+1}^{n-m}(i\,\ j).
\]

\begin{prop}
    The action of $x_i$ on $H_{2d}((\mathbb{S}^2)^{n-m})$ is equal to the action of
    \[
    \frac{1-s'_i}{2}(a_d-n+m+1+J_i) - \frac{1+s'_i}{2}\tilde{J}_i.
    \]
\end{prop}

\begin{proof}
    Let $U\subset\{1,\ldots,n-m\}$. We have
    \[
    J_i\cdot l_U =
    \begin{cases}
        \displaystyle\lvert \{ 1\leq j < i \mid j\in U\}\rvert l_U + \sum_{\substack{j=1\\j\not\in U}}^{i-1}l_{U\setminus\{i\}\cup\{j\}} & \text{if }i \in U,\\
        \displaystyle\lvert \{ 1\leq j < i \mid j\not\in U\}\rvert l_U + \sum_{\substack{j=1\\j\in U}}^{i-1}l_{U\setminus\{j\}\cup\{i\}} & \text{if }i \not\in U,
    \end{cases}
    \]
    and
    \[
    \tilde{J}_i\cdot l_U =
    \begin{cases}
        \displaystyle\lvert \{ i< j \leq n-m \mid j\in U\}\rvert l_U + \sum_{\substack{j=i+1\\j\not\in U}}^{n-m}l_{U\setminus\{i\}\cup\{j\}} & \text{if }i \in U,\\
        \displaystyle\lvert \{ i < j \leq n-m \mid j\not\in U\}\rvert l_U + \sum_{\substack{j=i+1\\j\in U}}^{n-m}l_{U\setminus\{j\}\cup\{i\}} & \text{if }i \not\in U.
    \end{cases}
    \]
    The proof is complete once we notice that $\frac{1-s'_i}{2}\cdot l_U = \delta_{i\in U}l_U$ and $\frac{1+s'_i}{2}\cdot l_U = \delta_{i\not\in U}l_U$.
\end{proof}

Nonetheless, the homology of the $\Delta$-Springer variety $\dspr_{(n-k,k),m}$ is not stable by the action of $W_{n-m}$, but only under the action of the degenerate affine Hecke algebra. In the extremal case $m=k$, the homology of the $\Delta$-Springer variety is equal to the homology of $(\mathbb{S}^2)^{n-k}$ in degree $0,2,\ldots,2k$ and is $\{0\}$ in all other degrees. Therefore, the description of the action of $W_{n-k}$ given in \cite{SW22} is equivalent to the action of the degenerate affine Hecke algebra $H_{n-k}$: since both actions are irreducible, the Jacobson density theorem implies that the two actions generate $\End(H_{2d}((\mathbb{S}^2)^{n-k}))$ .


\vspace*{1cm}



\begin{thebibliography}{BCDV13}


\bibitem[AcHe08]{achar-henderson}
P.~Achar and A.~Henderson, 
\newblock Orbit closures in the enhanced nilpotent cone,
\emph{Adv. Math.}, 219 (2008), no. 1, 27--62. 
\href{https://doi.org/10.1016/j.aim.2008.04.008}
{\path{10.1016/j.aim.2008.04.008}}


\bibitem[BoMc83]{BoMc}
\newblock Partial resolutions of nilpotent varieties.
\newblock In \emph{Analysis and topology on singular spaces, II, III (Luminy, 1981)}, \emph{Ast\'erisque}, Société Mathématique de France, Paris, Vol. 101 (1983), 23--74.

\bibitem[Bru08]{Br-cat-O}
J.~Brundan, 
\newblock Symmetric functions, parabolic category $\mathcal{O}$, and the Springer fiber.
\newblock \emph{Duke Math. J.} 143 (2008), no.1, 41--79.
\href{https://doi.org/10.1215/00127094-2008-015}
{\path{10.1215/00127094-2008-015}}.

\bibitem[BrSt10]{BS2}
J.~Brundan and C.~Stroppel, 
\newblock Highest weight categories arising from Khovanov’s diagram algebra II: Koszulity.
\newblock \emph{Transform. Groups} 15 (2010), no.1, 1--45.
\href{https://doi.org/10.1007/s00031-010-9079-4}
{\path{10.1007/s00031-010-9079-4}}.

\bibitem[BrSt11a]{BS1}
J.~Brundan and C.~Stroppel, 
\newblock Highest weight categories arising from Khovanov’s diagram algebra I: cellularity.
\newblock \emph{Mosc. Math. J.} 11(2011), no.4, 685--722, 821--822, 2011.
\href{https://doi.org/10.17323/1609-4514-2011-11-4-685-722}
{\path{10.17323/1609-4514-2011-11-4-685-722}}.

\bibitem[BrSt11b]{BS3}
J.~Brundan and C.~Stroppel, 
\newblock Highest weight categories arising from Khovanov’s diagram algebra III: category $\cO$.
\newblock \emph{Represent. Theory} 115 (2011), 170--243.
\href{https://doi.org/10.1090/S1088-4165-2011-00389-7}
{\path{10.1090/S1088-4165-2011-00389-7}}.

\bibitem[BrSt12]{BS4}
J.~Brundan and C.~Stroppel, 
\newblock Highest weight categories arising from Khovanov’s diagram algebra IV: the general linear supergroup.
\newblock \emph{J. Eur. Math. Soc.} 14 (2012), no.2, 373--419.
\href{https://doi.org/10.4171/JEMS/306}
{\path{10.4171/JEMS/306}}.

\bibitem[CaKa08]{CK08}
S.~Cautis and J.~Kamnitzer,
\newblock Knot homology via derived categories of coherent sheaves. {I}. {T}he $\mathfrak{sl}$(2)-case.
\newblock \emph{Duke Math. J.} 142 (2008), no. 3, 511--588.
\href{https://doi.org/10.1215/00127094-2008-012}
{\path{10.1215/00127094-2008-012}}.



\bibitem[Che05]{che05}
I.~Cherednik,
\newblock Double affine Hecke algebras.
\newblock \emph{London Math. Soc. Lecture Note Ser.,} vol. 319 (Cambridge University Press, Cambridge, 2005).
\href{https://doi.org/10.1017/CBO9780511546501}
{\path{10.1017/CBO9780511546501}}.

\bibitem[CDV11]{CDV}
A.~Cox and M.~De Visscher,
\newblock Diagrammatic Kazhdan-Lusztig theory for the (walled) Brauer algebra.
\newblock \emph{J. Algebra} 340 (2011), 151--181.
\href{https://doi.org/10.1016/j.jalgebra.2011.05.024}
{\path{10.1016/j.jalgebra.2011.05.024}}.

\bibitem[CDVDM08]{CDVDM}
A.~Cox, M.~De Visscher, S.~Doty and P.~Martin,
\newblock On the blocks of the walled Brauer algebra.
\newblock \emph{J. Algebra} 320 (2008), 169--212.
\href{https://doi.org/10.1016/j.jalgebra.2008.01.026}
{\path{10.1016/j.jalgebra.2008.01.026}}.


\bibitem[EhSt16a]{EhSt16}
M.~Ehrig and C.~Stroppel,
\newblock $2$-row Springer fibres and Khovanov diagram algebras for type D.
\newblock \emph{Canad. J. Math.} 68 (2016), no. 6, 1285--1333.
\href{https://doi.org/10.4153/CJM-2015-051-4 }
{\path{10.4153/CJM-2015-051-4 }}.

\bibitem[EhSt16b]{EhSt-diagrammatic}
M.~Ehrig and C.~Stroppel,
\newblock Diagrammatic description for the categories of perverse sheaves on isotropic Grassmannians.
\newblock \emph{Selecta Math. (N.S.)} 22 (2016), no.3, 1455--1536.
\href{https://doi.org/10.1007/s00029-015-0215-9}
{\path{10.1007/s00029-015-0215-9}}.

\bibitem[EhSt16c]{EhSt-koszul}
M.~Ehrig and C.~Stroppel,
\newblock Koszul gradings on Brauer algebras.
\newblock \emph{Int. Math. Res. Not. IMRN} (2016), no.13, 3970--4011.
\href{https://doi.org/10.1093/imrn/rnv267}
{\path{10.1093/imrn/rnv267}}.


\bibitem[EhSt17]{EhSt17}
M.~Ehrig and C.~Stroppel,
\newblock On the category of finite-dimensional representations of {${\rm OSp}(r|2n)$}: {P}art {I}.
\newblock In \emph{Representation theory—current trends and perspectives}, 109--170, \emph{Eur. Math. Soc.}, Zürich (2017).
\href{https://doi.org/10.4171/171}
{\path{10.4171/171}}.

\bibitem[EhSt]{EhSt}
M.~Ehrig and C.~Stroppel,
\newblock Deligne categories and representations of {${\rm OSp}(r|2n)$}.
\newblock Available at \href{http://www.math.uni-bonn.de/ag/stroppel/OSPII.pdf}{\path{http://www.math.uni-bonn.de/ag/stroppel/OSPII.pdf}}.

\bibitem[FrMe10]{FrMe10}
L.~Fresse and A.~Melnikov,
\newblock On the singularity of the irreducible components of a Springer fiber in $\mathfrak{sl}_{n}$.
\newblock \emph{Selecta Math. (N.S.)} 16 (2010), no.3, 393--418. \href{https://doi.org/10.1007/s00029-010-0025-z}
{\path{10.1007/s00029-010-0025-z}}.

\bibitem[Fun03]{Fun03}
F.~Fung,
\newblock On the topology of components of some {S}pringer fibers and their relation to {K}azhdan-{L}usztig theory. 
\newblock \emph{Adv. Math.} 178 (2003), no. 2, 244--276.
\href{https://doi.org/10.1016/S0001-8708(02)00072-5}
{\path{10.1016/S0001-8708(02)00072-5}}.



\bibitem[GiGr24]{GiGr23}
M.~Gillespie and S.~Griffin,  
\newblock Cocharge and skewing formulas for $\Delta$-Springer modules and the Delta Conjecture. 
\newblock \emph{Int. Math. Res. Not. IMRN}, rnae090.
\newblock \href{https://doi.org/10.1093/imrn/rnae090}{\path{10.1093/imrn/rnae090}}.

\bibitem[GLW24]{GLW21}
S.~Griffin, J.~Levinson and A.~Woo, 
\newblock Springer fibers and the Delta conjecture at $t=0$. 
\newblock \emph{Adv. Math.} 439 (2024), 109491.
\href{https://doi.org/10.1016/j.aim.2024.109491}
{\path{10.1016/j.aim.2024.109491}}.

\bibitem[HRS18]{HRS18}
J.~Haglund, B.~Rhoades and M.~Shimozono,
\newblock Ordered set partitions, generalized coinvariant algebras, and the Delta Conjecture.
\newblock \emph{Adv. Math.} 329 (2018), 851--915. \href{https://doi.org/10.1016/j.aim.2018.01.028}
{\path{10.1016/j.aim.2018.01.028}}.

\bibitem[ILW22]{ILW22}
M.~S.~Im, C.-J.~Lai and A.~Wilbert,
\newblock Irreducible components of two-row Springer fibers for all classical types.
\newblock \emph{Proc. Amer. Math. Soc.} 150 (2022), no.6, 2415--2432. \href{https://doi.org/10.1090/proc/15965}
{\path{10.1090/proc/15965}}.

\bibitem[JaPe79]{jp79}
G.~D.~James and M.~H.~Peel, 
\newblock Specht series for skew representations of symmetric groups.
\newblock \emph{J. Algebra} 56 (1979), no.2, 343--364.
\href{https://doi.org/10.1016/0021-8693(79)90342-9}
{\path{10.1016/0021-8693(79)90342-9}}.

\bibitem[Kat06]{Kato06}
S. Kato, 
\newblock An exotic Springer correspondence for symplectic groups.
\newblock 2006.
\newblock \url{https://arxiv.org/abs/math/0607478}.

\bibitem[Kho02]{Kh-arc}
M.~Khovanov, 
\newblock A functor-valued invariant of tangles
\newblock \emph{Algebr. Geom. Topol.} 2 (2002), 665--741.
\href{https://doi.org/10.2140/agt.2002.2.665}
{\path{10.2140/agt.2002.2.665}}.


\bibitem[Kho04]{Kho04}
M.~Khovanov, 
\newblock Crossingless matchings and the cohomology of $(n,n)$ Springer varieties.
\newblock \emph{Commun. Contemp. Math.} 6 (2004), no.4, 561--577.
\href{https://doi.org/10.1142/S0219199704001471}
{\path{10.1142/S0219199704001471}}.


\bibitem[LNV21]{LaNaVa}
A.~Lacabanne, G.~Naisse and P.~Vaz, 
\newblock Tensor product categorifications, Verma modules and the blob $2$-category.
\newblock \emph{Quantum Topol.} 12 (2021), no.4, 705--812.
\href{https://doi.org/10.4171/qt/156}
{\path{10.4171/qt/156}}.

\bibitem[MaSa94]{MaSa-blob}
P.~Martin and H.~Saleur, 
\newblock The blob algebra and the periodic Temperley-Lieb algebra.
\newblock \emph{Lett. Math. Phys.} 30 (1994), no.3, 189--206.
\href{https://doi.org/10.1007/BF00805852}
{\path{10.1007/BF00805852}}.

\bibitem[NRS18]{NaRoSa}
V.~Nandakumar, D.~Rosso, and N.~Saunders, 
\newblock Irreducible components of exotic {S}pringer fibres.
\newblock \emph{J. Lond. Math. Soc.} (2) 98 (2018), no.3, 609--637.
\href{https://doi.org/10.1112/jlms.12152}
{\path{10.1112/jlms.12152}}.

\bibitem[Ram03]{ram03}
A.~Ram, 
\newblock Skew shape representations are irreducible.
\newblock In \emph{Combinatorial and geometric representation theory (Seoul,
2001)}, \emph{Contemp. Math.} 325 (2003), 161--189.
\href{https://doi.org/10.1090/conm/325}
{\path{10.2140/10.1090/conm/325}}.

\bibitem[Rus11]{russell}
H.~Russell, 
\newblock A topological construction for all two-row {S}pringer varieties.
\newblock \emph{Pacific J. Math.} 253 (2011), no. 1, 221--255.
\href{https://doi.org/10.2140/pjm.2011.253.221}
{\path{10.2140/pjm.2011.253.221}}.

\bibitem[RuTy11]{RuTy11}
H.~Russell and J.~Tymoczko, 
\newblock Springer representations on the Khovanov Springer varieties.
\newblock \emph{Math. Proc. Cambridge Philos. Soc.} 151 (2011), no. 1, 59--81.
\href{https://doi.org/10.1017/S0305004111000132}
{\path{10.1017/S0305004111000132}}.


\bibitem[SaWi22]{SW22}
N.~Saunders and A.~Wilbert, 
\newblock Exotic {S}pringer fibers for orbits corresponding to one-row bipartitions.
\newblock \emph{Transform. Groups} 27 (2022), no. 3, 1111--1147.
\href{https://doi.org/10.1007/s00031-020-09613-0}
{\path{10.1007/s00031-020-09613-0}}.

\bibitem[Sch12]{Sch12}
G.~Sch{\"a}fer, 
\newblock A graphical calculus for 2-block {S}paltenstein varieties.
\newblock \emph{Glasg. Math. J.} 54 (2012), no. 2, 449--477.
\href{https://doi.org/10.1017/S0017089512000110}
{\path{10.1017/S0017089512000110}}.

\bibitem[Spr76]{spr76}
T.~A.~Springer, 
\newblock Trigonometric sums, Green functions of ﬁnite groups and representations of Weyl
groups.
\newblock \emph{Invent. Math.} 36 (1976), no.1, 173--207.
\href{https://doi.org/10.1007/BF01390009}
{\path{10.1007/BF01390009}}.

\bibitem[Spr78]{spr78}
T.~A.~Springer, 
\newblock A construction of representations of Weyl groups.
\newblock \emph{Invent. Math.} 44 (1978), no.3, 279--293.
\href{https://doi.org/10.1007/BF01403165}
{\path{10.1007/BF01403165}}.

\bibitem[Sta12]{stanley}
R.~Stanley,
\newblock Enumerative combinatorics. Volume 1.
\newblock \emph{Cambridge Stud. Adv. Math.} 49 (Cambridge University Press, Cambridge, 2012).

\bibitem[Str09]{St-cat-O}
C.~Stroppel, 
\newblock Parabolic category $\cO$, perverse sheaves on Grassmannians, Springer fibres and Khovanov homology.
\newblock \emph{Compos. Math.} 145 (2009), no.4, 954--992.
\href{https://doi.org/10.1112/S0010437X09004035}
{\path{10.1112/S0010437X09004035}}.

\bibitem[StWe12]{SW12}
C.~Stroppel and B.~Webster,
\newblock 2-block {S}pringer fibers: convolution algebras and coherent sheaves.
\newblock \emph{Comment. Math. Helv.} 87 (2012), no.2, 477--520.
\href{https://doi.org/10.4171/CMH/261}
{\path{10.4171/CMH/261}}

\bibitem[StWi19]{StWi19}
C.~Stroppel and A.~Wilbert,
\newblock Two-block Springer fibers of types C and D: a diagrammatic approach to Springer theory.
\newblock \emph{Math. Z.} 292 (2019), no.3-4, 1387--1430.
\href{https://doi.org/10.1007/s00209-018-2161-7}
{\path{10.1007/s00209-018-2161-7}}

\bibitem[Suz98]{suz98}
T.~Suzuki, 
\newblock Representations of degenerate affine Hecke algebra and $\mathfrak{gl}_n$.
\newblock \emph{Adv. Stud. Pure Math.} 28 (1998), 343--372.
\href{https://doi.org/10.2969/aspm/02810343}
{\path{10.2969/aspm/02810343}}.

\bibitem[Weh09]{Weh09}
S.~Wehrli,
\newblock A remark on the topology of $(n,n)$ {S}pringer varieties. 
\newblock 2009.
\newblock \url{http://arxiv.org/abs/0908.2185}.

\bibitem[Wil18]{Wil18}
A.~Wilbert, 
\newblock Topology of two-row Springer fibers for the even orthogonal and symplectic group.
\newblock \emph{Trans. Amer. Math. Soc.} 370 (2018), no. 4, 2707--2737.
\href{https://doi.org/10.1090/tran/7194}
{\path{10.1090/tran/7194}}.


   
\end{thebibliography}
\end{document}